\documentclass[reqno,10pt]{amsart}

\usepackage{amsmath, amsthm}
\usepackage{pdfsync}
\usepackage{parskip}
\usepackage{hyperref}
\usepackage{mathtools}
\usepackage{mdframed}
\usepackage{mathrsfs}
\usepackage{tikz}
\numberwithin{equation}{section}
\newtheorem{theorem}{Theorem}[section]
\newtheorem{prop}[theorem]{Proposition}

\newtheorem{lem}[theorem]{Lemma}

\theoremstyle{remark}
\newtheorem{remark}[theorem]{Remark}

\def\M{\Sigma}
\def\p{\mathsf{p}}

\def\T{{\sf T}}

\def\R{\mathbb{R}}
\def\C{\mathbb{C}}

\def\N{\mathbb{N}}
\def\Nz{\mathbb{N}_0}
\def\A{\mathfrak{A}}
\def\Ae{\mathfrak{A}_{\frak{e}}}
\def\K{\frak{K}}
\def\Ke{\frak{K}_{\frak{e}}}

\def\L{\mathcal{L}}
\def\Lis{\mathcal{L}{\rm{is}}}
\def\F{\mathfrak{F}}
\def\bF{\boldsymbol{\mathfrak{F}}}
\def\bB{\boldsymbol{\frak{B}}}

\def\B{\mathbb{B}}
\def\Br{\mathbb{B}(0,r_0)}
\def\O{\mathsf{O}}
\def\Oe{\mathsf{O}_\frak{e}}
\def\Oei{\mathsf{O}_{\frak{e}_i}}
\def\Ok{\mathsf{O}_{\kappa}}
\def\Oek{\mathsf{O}_{\frak{e},\kappa}}
\def\Ot{\mathsf{O}_{{\kappa_c}}}
\def\Oet{\mathsf{O}_{\frak{e},{\kappa_c}}}
\def\vpk{\varphi_{\kappa}}
\def\psk{\psi_{\kappa}}
\def\vpt{\varphi_{{\kappa_c}}}
\def\pst{\psi_{{\kappa_c}}}
\def\vpet{\varphi_{\frak{e},{\kappa_c}}}
\def\pset{\psi_{\frak{e},{\kappa_c}}}
\def\vpek{\varphi_{\frak{e},\kappa}}
\def\psek{\psi_{\frak{e},\kappa}}
\def\Q{\mathsf{Q}^{m}}

\def\Qa{\mathsf{Q}_a}

\def\kbt{\varphi^{\ast}_{{\kappa_c}}}
\def\kft{\psi^{\ast}_{{\kappa_c}}}
\def\kbet{\varphi^{\ast}_{\frak{e},{\kappa_c}}}
\def\kfet{\psi^{\ast}_{\frak{e},{\kappa_c}}}

\def\kf{\psi^{\ast}_{\kappa}}
\def\keb{\varphi^{\ast}_{\frak{e},\kappa}}
\def\kef{\psi^{\ast}_{\frak{e},\kappa}}

\def\pk{\pi_\kappa}
\def\pek{\pi_{\frak{e},\kappa}}

\def\tm{\theta^{\ast}_{\mu}}
\def\tmi{\theta_{\ast}^{\mu}}
\def\te{\bar{\theta}^{\ast}_{\eta}}
\def\Tm{{\Theta}^{\ast}_{\mu}}
\def\Tmi{{\Theta}^{\mu}_{\ast}}
\def\tlm{\theta^{\ast}_{\lambda,\mu}}
\def\Tlm{{\Theta}^{\ast}_{\lambda,\mu}}
\def\Tlmi{{\Theta}_{\ast}^{\lambda,\mu}}
\def\Tu{{T}_{\mu}}
\def\Tui{{T}^{-1}_{\mu}}
\def\tTu{\tilde{T}_{\mu}}

\def\rl{{\varrho}^{\ast}_{\lambda}}

\def\tme{\theta^{\ast}_{\mu,\eta}}
\def\tmei{\theta_{\ast}^{\mu,\eta}}
\def\Tme{{\Theta}^{\ast}_{\mu,\eta}}
\def\Tmei{{\Theta}^{\mu,\eta}_{\ast}}
\def\tlme{\theta^{\ast}_{\lambda,\mu,\eta}}
\def\Tlme{{\Theta}^{\ast}_{\lambda,\mu,\eta}}
\def\Tlmei{{\Theta}_{\ast}^{\lambda,\mu,\eta}}
\def\Tue{{T}_{\mu,\eta}}
\def\Tuei{{T}^{-1}_{\mu,\eta}}
\def\tTue{\tilde{T}_{\mu,\eta}}

\def\tAme{\tilde{\mathcal{A}}_{\mu,\eta}}
\def\Ame{{\mathcal{A}}_{\mu,\eta}}

\def\ej{\mathbb{E}(J)}
\def\zej{\prescript{}{0}{\mathbb{E}(J)}}
\def\ea{\mathbb{E}_{1}(J)}
\def\zea{\prescript{}{0}{\mathbb{E}_{1}(J)}}
\def\eb{\mathbb{E}_{2}(J)}
\def\zeb{\prescript{}{0}{\mathbb{E}_{2}(J)}}

\def\fj{\mathbb{F}(J)}
\def\zfj{\prescript{}{0}{\mathbb{F}(J)}}
\def\fa{\mathbb{F}_{1}(J)}
\def\fb{\mathbb{F}_{2}(J)}
\def\fc{\mathbb{F}_{3}(J)}
\def\zfb{\prescript{}{0}{\mathbb{F}_{2}(J)}}
\def\zfc{\prescript{}{0}{\mathbb{F}_{3}(J)}}
\def\fd{\mathbb{F}_{4}}
\def\gb{\mathbb{G}_{2}(J)}
\def\gc{\mathbb{G}_{3}(J)}
\def\ca{\mathbb{C}_{1}(J)}
\def\cb{\mathbb{C}_{2}(J)}

\def\Re{\mathcal{R}}
\def\Rek{\mathcal{R}_{\kappa}}
\def\Rc{\mathcal{R}^c}
\def\Rck{\mathcal{R}^c_{\kappa}}
\def\Ree{\mathcal{R}_{\frak{e}}}
\def\Reek{\mathcal{R}_{\frak{e},\kappa}}
\def\Rce{\mathcal{R}^c_{\frak{e}}}
\def\Rcek{\mathcal{R}^c_{\frak{e},\kappa}}

\def\id{{\rm{id}}}
\def\supp{{\rm{supp}}}

\def\cp{{\rm{cp}}}
\def\Dfi{{\rm{Diff}^{\hspace{.1em}\infty}}}
\def\im{{\rm{im}}}
\def\tr{{\rm{tr}}}

\def\Xk{\mathbb{X}_{\frak{e},\kappa}}

\def\cH{\mathcal{H}}
\def\bz{\hat{\beta}_0}

\def\vge{\varsigma_{\frak{e}}}

\begin{document}

\title[regularity of the interface of the Stefan problem]
{On the regularity of the interface of a thermodynamically consistent two-phase Stefan problem with surface tension}

\author[J.~Pr\"uss]{Jan Pr\"uss}
\address{Institut f\"ur Mathematik,
         Martin-Luther-Universit\"at Halle-Wittenberg,
         60120 Halle, Germany}
\email{jan.pruess@mathematik.uni-halle.de}

\author[Y. Shao]{Yuanzhen Shao}
\address{Department of Mathematics,
         Vanderbilt University,  
         Nashville, TN 37240, USA}
\email{yuanzhen.shao@vanderbilt.edu}

\author[G. Simonett]{Gieri Simonett}
\address{Department of Mathematics\\
         Vanderbilt University \\
         Nashville, TN~37240, USA}
\email{gieri.simonett@vanderbilt.edu}

\thanks{The research of the third author was partially supported by NSF DMS-1265579.}

\begin{abstract}
We study the regularity of the free boundary arising in a thermodynamically consistent two-phase Stefan problem with surface tension by means of a family of parameter-dependent diffeomorphisms, $L_p$-maximal regularity theory, and the implicit function theorem. 
\end{abstract}
\maketitle

\section{\bf Introduction}

The main objective of this article is to develop a technique relying on a family of parameter-dependent diffeomorphisms, maximal regularity theory, and the implicit function theorem to prove regularity of moving interfaces occurring in free boundary problems. As an application, we prove that the moving interface in a thermodynamically consistent two-phase Stefan problem is jointly $C^k$-smooth in time and space, for $k\in \N\cup \{\infty,\omega\}$ with $\omega$ being the symbol of real analyticity, as long as several physical quantities, that is, the coefficients of the heat conductivity, kinetic undercooling, and the free energy, enjoy appropriate regularity assumptions.


The idea of establishing regularity of solutions to differential equations by means of the implicit function theorem in conjunction with a translation argument was first introduced by S.B.~Angenent in \cite{Ange88} to prove   analyticity of the free boundary in one dimensional porous medium equations, and has proved itself a useful tool later in many publications. See for example \cite{Ange902,EscPruSim0302, EscSim96, Lunar95, PruSim07}. More precisely, to study the regularity of the solution to a differential equation, one introduces parameters representing translation in space and time into the solution to the given differential equation. Then one studies the parameter-dependent equation satisfied by this transformed solution. The implicit function theorem yields the smooth dependence upon the parameters of the solution to the parameter-dependent problem. This regularity property is then inherited by the original solution. An advantage of this technique is reflected by its power to prove analyticity of solutions to differential equations, which cannot be attained through the classical method of bootstrapping.

A well-known approach to free boundary problems is to transform the original problem with a moving boundary, or separating interface, which we denote by $\Gamma(t)$, into one with a fixed reference manifold $\Sigma$ by means of the Hanzawa transformation, see \cite{Han81}. Then the problem of establishing the regularity of the free boundary $\Gamma(t)$ is transferred into establishing the regularity of the height function parameterizing $\Gamma(t)$ over $\Sigma$.
However, applying the aforementioned translation technique to the height function on the surface $\Sigma$ causes an essential challenge, considering for instance the usual translation $(t,x)\mapsto(t+\lambda,x+\mu)$, because of the global nature of these translations. 
Hence we desire an alternative that only shifts the variables ``locally".
The idea of localizing translations was first introduced by J.~Escher, J.~Pr\"uss and G.~Simonett in \cite{EscPruSim03} to study regularity of solutions to elliptic and parabolic equations in Euclidean space. The basic building block of \cite{EscPruSim03} consists of rescaling translations by some cutoff function. This technique was later generalized in Y.~Shao \cite{ShaoPre} to introduce a family of parameter-dependent diffeomorphisms acting on functions or tensor fields on Riemannian manifolds by means of a smooth atlas in order to study the regularity of solutions to geometric evolution equations.
But in view of the physical quantities in the bulk phases adjacent to $\Gamma(t)$, e.g., the temperature function in the case of phase transitions, or the velocity and the pressure field in the case of two-phase fluids, we need introduce a localized translation not only on the fixed reference surface $\Sigma$, but also in a neighborhood of $\Sigma$. This adds one more degree of complexity to the aforementioned technique for geometric flows. In Section~3, we will build up a complete theory of parameter-dependent diffeomorphisms for free boundary problems.


Free boundary problems form an important field of applied analysis. They deal with solving partial differential equations in a given domain, a part of whose boundary is a priori unknown. That portion of the boundary is called the free boundary or the moving boundary. In addition to the standard boundary conditions that are
needed in order to solve the prescribed partial differential equations, an additional condition must be imposed at the free boundary. One then seeks to determine both the free boundary and the solution of the differential equation. This field has drawn great attention over decades due to its applications to physics, chemistry, medicine, material science, and so forth.

Great progress has been seen in the studies of regularity of free boundaries during the past half century. Many mathematicians have made  contributions  to the theory of free boundary problems, among them H.W.~Alt, H.~Berestycki, L.A.~Caffarelli, A.~Friedman, D.~Kinderlehrer, L.~Nirenberg, J.~Spruck, G.~Stampacchia  \cite{AltCaf81,AltCafFdm84, BerCaf90, Caf77, Caf87, KinNir77, KinNir78, KinNirSpr79, KinSta80}, E.~DiBenedetto \cite{Dib82}. We also refer the reader to \cite{FmdRei88,KinSta80} for a historical account of the field up to the 1980s. 


The Stefan problem, arguably, is the most studied free boundary problem,
with over 1200 mathematical publications devoted to the topic. It was first introduced in 1889 by 
J.~Stefan. We refer the reader to the books \cite{Mei92, Rub71, Vis96} for further information.
The Stefan problem  describes phase transitions
in liquid-solid systems and accounts for heat diffusion and
exchange of latent heat in a homogeneous medium, wherein the liquid and solid phases are separated by a closed moving interface $\Gamma(t)$. 
The basic physical law governing this process is conservation of energy. 
To be more precise, let $\Omega\subset \R^{m+1}$ be a bounded domain of class $C^{2}$, $m\geq 1$.
$\Omega$ is occupied by a material that can undergo phase changes: at time $t$, phase $i$ occupies
the subdomain $\Omega_i(t)$ of
$\Omega$, respectively, with $i=1,2.$
We assume that $\partial \Omega_1(t)\cap\partial \Omega=\emptyset$; this means
that no {\em boundary contact} can occur.
The closed compact hypersurface $\Gamma(t):=\partial \Omega_1(t)\subset \Omega$
forms the interface between the phases. 
Then the {\em Stefan problem with surface tension}, or the \emph{Stefan problem with Gibbs-Thomson correction},  can be formulated as follows:
\begin{equation}
\label{S1: Stefan-0}
\left\{\begin{aligned}
\partial_t \vartheta-d\Delta \vartheta&=0 &&\text{in}&&\Omega\setminus\Gamma(t),\\
\partial_{\nu_{\partial\Omega}}\vartheta&=0 &&\text{on}&&\partial \Omega, \\
V(t)-[\![d\partial_{\nu_\Gamma}\vartheta]\!]&=0 &&\text{on}&&\Gamma(t),\\
\vartheta+\sigma \cH &=0   &&\text{on}&&\Gamma(t), \\
 \vartheta(0)=\vartheta_0, \quad  \Gamma(0) &=\Gamma_0. &&
\end{aligned}\right.
\end{equation}
Here
 $\vartheta(t)=\vartheta_1(t)\chi_{\Omega_1(t)} + \vartheta_2(t)\chi_{\Omega_2(t)}$, where $\vartheta_i$ denotes the  relative temperature distribution in phase $i$,
 $\nu_\Gamma(t)$ the outer normal field of $\partial\Omega_1(t)$,
 $V(t)$ the normal velocity of $\Gamma(t)$,
 $\cH(t)=\cH(\Gamma(t))=-{\rm div}_{\Gamma(t)} \nu_\Gamma(t)/m$ the mean curvature of $\Gamma(t)$, and
 $[\![v]\!]=v_2|_{\Gamma(t)}-v_1|_{\Gamma(t)}$ the jump of a quantity $v$ across $\Gamma(t)$.
The sign of the mean curvature $\cH$ is chosen to be negative at a point $x\in\Gamma$ if
$\Omega_1\cap \B(x,r)$ is convex for some sufficiently small $r>0$. Thus if $\Omega_1$ is a ball
of radius $R$ then $\cH=-1/R$ for its boundary $\Gamma$. The function $d$ agrees with a positive constant $d_i$ in $\Omega_i(t)$, which stand for the heat conductivities in different phases. 
The condition that $V(t)=[\![d\partial_{\nu}\vartheta]\!]$ is caused by the law of conservation of energy and is usually called the {\em Stefan condition}. 

If a pair $(\vartheta,\Gamma)$ solves the system of equations (1.1), 
$\Gamma$ is called the free boundary to the Stefan problem.
In the case that  the condition $\vartheta+\sigma \cH =0$ is replaced by
\begin{equation}
\label{S1: GT}
\vartheta  =0 \quad\text{on}\quad\Gamma(t),
\end{equation}
i.e., if $\sigma=0$, the resulting problem is usually referred to as the \emph{classical Stefan problem}.  
Here we mention the monographs  by B.~Chalmers~\cite[Chapter~1]{Cha77}, 
P.~Hartman~\cite{Har73},
A.~Visintin~\cite{Vis96},
and the research papers by
G.~Caginalp~\cite{Cag86},
M.E.~Gurtin \cite{Gur86}-\cite{Gur88-2},
J.S.~Langer~\cite{Lan80},
W.W.~Mullins and R.F.~Sekerka~\cite{MuSe63, MuSe64}, 
A. Visintin~\cite{Vis93},
where this law has been motivated and derived based on various mathematical and physical principles.
As has been explained in \cite[Section 1]{PruSimZac12}, see also \cite{PruSimPre}, condition \eqref{S1: GT}
can be understood as a first order approximation of the relation
\begin{equation}
\label{S1: GT-consistent}
[\![\psi(\theta)]\!]+\sigma \cH=0\quad\text{on}\quad \Gamma(t)
\end{equation}
around the melting temperature, where 
$\theta$ is the absolute temperature and $\psi$ is the free energy of the system.

In the absence of one of $\vartheta_i$, that is, one of $\vartheta_i$ is identically equal to some constant, the problem is called the one-phase Stefan problem, otherwise it is alluded to as the two-phase Stefan problem.

In the one-dimensional case, different aspects of the \emph{classical Stefan problem} were extensively studied before the 1980s.
In higher dimensions, global existence and uniqueness of weak solutions was first established by S.L.~Kamenomostskaja in \cite{Kam61}, see also A.~Friedman \cite{Fmd68}. Existence of local classical solutions was obtained by E.I.~Hanzawa \cite{Han81} and A.M.~Me\v{i}rmanov \cite{Mei80}, provided that the initial data are smooth enough and  satisfy some higher order compatibility conditions.

In the case of the one-phase {\em classical Stefan problem}, regularity of the free boundary for  weak solutions is studied in \cite{Caf77,Caf78,FmdKin74,KinNir77,KinNir78}. The formulation of the problem as a parabolic variational inequality was initiated by G.~Duvaut in \cite{Duv73}. After that, it has been applied in \cite{FmdKin74} to get the Lipschitz continuity of the free boundary under some conditions on the initial geometry and the given data. Later in a paper of L.A.~Caffarelli \cite{Caf77}, the author proved that the free boundary is $C^1$ in space and time and the temperature is $C^2$ up to the free boundary near a density point for the states set, namely the solid phase. In particular, if the free boundary is Lipschitz continuous then it is actually $C^1$. In \cite{Caf78}, the author identified conditions to guarantee a locally Lipschitz free boundary. Almost at the same time, D.~Kinderlehrer and L.~Nirenberg \cite{KinNir77} established via the partial hodograph and Legendre transformations in conjunction with a variational inequality argument smoothness of the free boundary as well as the solution. The conditions in \cite{KinNir77} can be verified by the results in \cite{Caf77}. In a subsequent paper \cite{KinNir78}, the same authors showed under some integral conditions by the same technique that the free boundary and the solution are analytic in the space variables and in the second Gevrey class for the time variable. Regularization of the free boundary for large time was obtained by H.~Matano in \cite{Mat82}, i.e., any weak solution eventually becomes smooth. Continuity of the temperature distribution for weak solutions is proved in \cite{CafFmd79}.

Regularity of the free boundary for weak solutions or   viscosity solutions has been explored in \cite{AthCaf96,AthCafSal98,Koc98} for the multi-dimensional two-phase Stefan problem based on a non-degeneracy condition. The non-degeneracy condition states, roughly speaking, that the heat fluxes are not vanishing simultaneously on the free boundary, which is automatically satisfied when the free boundary is regular enough, say belongs to the H\"older class of order more than $1$. Assuming that the free boundary is locally Lipschitz continuous, it was proven in \cite{AthCaf96} that the solution is classical, i.e., the free boundary and the viscosity solution are $C^1$. Later the author showed that the non-degeneracy condition can be replaced by some flatness condition. In the presence of the non-degeneracy condition, H.~Koch verified that the $C^1$ free boundary and the temperature are in fact smooth. Through a different approach by imposing a mild regularity condition on the initial data, it was shown by J.~Pr\"uss, J.~Saal and G.~Simonett in \cite{PruSim07} that the free boundary and the temperature are analytic under the assumption that the free boundary can be expressed as the graph of some function. Continuity of the temperature distribution for weak solutions was obtained in \cite{CafEva83,Dib80,Dib82,Zie82}.

Although the Stefan problem with the
{\em Gibbs-Thomson correction} \eqref{S1: GT} has been around
for many decades, only few analytical results 
concerning existence of solutions can be found in the literature. A.~Friedman and F.~Reitich \cite{FmdRei91} considered the case
with small surface tension $0<\sigma\ll 1$ 
and linearized the problem about $\sigma=0$. 
Assuming the existence of a smooth solution for the case $\sigma=0$, that is, 
for the {\em classical Stefan problem}, the authors proved existence
and uniqueness of a weak solution for the 
\textit{linearized} problem and
then investigated the effect of small surface tension on
the shape of $\Gamma(t)$.
Existence of long time weak solutions was first established by Luckhaus~\cite{Luc90},
using a discretized problem and a capacity-type estimate for approximating solutions. 
The weak solutions obtained have a sharp interface, but are highly non-unique. 
See also R\"oger~\cite{Roe04}, and Almgren and Wang~\cite{AlWa00}.  
Existence of classical solutions, but without uniqueness, was proved by E.V.~Radkevich in \cite{Rad91}. 
In A.M. Me\v{i}rmanov~\cite{Mei94}, 
the way in which a spherical ball of ice 
in a supercooled fluid melts down was investigated. 
It was presented by J.~Escher, J.~Pr\"uss and G.~Simonett in \cite{EscPruSim0302} that there exists a unique local strong solution and the free boundary immediately regularizes to be analytic in space and time provided that the initial data satisfy some mild regularity assumptions. This result is based on the assumption that the free boundary is given by the graph of some function. 
In J.~Pr\"uss and G.~Simonett~\cite{PruSim08}, linearized stability and instability of equilibria was investigated.
Nonlinear stability results were obtained by 
M.~Had{\v{z}}i{\'c}, Y.~Guo~\cite{Had12,HadGuo10} and J.~Pr\"uss, G.~Simonett and R.~Zacher~\cite{PruSimZac12}.

In this paper we consider a general model for phase transitions, formulated in \cite{PruSimZac12}, that is thermodynamically consistent,
see \cite{Gur07} and \cite{Ish06} for related work. 
It involves the thermodynamic quantities of absolute temperature, free energy, internal energy, and entropy,
and is complemented by constitutive equations for the free energies and the heat fluxes in the bulk regions.
An important assumption is that there be no entropy production on the interface.
In particular, the interface is assumed to carry no mass and no energy
except for surface tension. 

To be more precise, 
we choose $\Omega$, $\Omega_i(t)$ and $\Gamma(t)$ as in \eqref{S1: Stefan-0}.
By the thermodynamically consistent two-phase {\em Stefan problem with surface tension}, possibly with kinetic undercooling, we mean the following problem:
find a family of closed compact hypersurfaces $\{\Gamma(t)\}_{t\geq0}$ contained in $\Omega$
and an appropriately smooth function $\theta:\R_+\times\bar{\Omega}\rightarrow\R$ such that
\begin{equation}
\label{stefan}
\left\{\begin{aligned}
\kappa (\theta)\partial_t \theta-{\rm div}(d(\theta)\nabla \theta)&=0 &&\text{in}&&\Omega\setminus\Gamma(t),\\
\partial_{\nu_\Omega} \theta &=0 &&\text{on}&&\partial \Omega, \\
[\![\theta]\!]&=0 &&\text{on}&&\Gamma(t),\\
[\![\psi(\theta)]\!]+\sigma \cH &=\gamma(\theta) V &&\text{on}&&\Gamma(t), \\
[\![d(\theta)\partial_{\nu_\Gamma} \theta]\!] &=(l(\theta)-\gamma(\theta)V) V &&\text{on}&&\Gamma(t),\\
\theta(0)=\theta_0, \quad  \Gamma(0)&=\Gamma_0. &&
\end{aligned}\right.
\end{equation}
Here  $\theta$ denotes the (absolute) temperature. 
Several quantities are derived from the free energies $\psi_i(\theta)$ as follows:
\begin{itemize}
\item
 $\epsilon_i(\theta)= \psi_i(\theta)+\theta\eta_i(\theta)$, the internal energy in phase $i$,
\item
 $\eta_i(\theta) =-\psi_i^\prime(\theta)$, the entropy,
\item
 $\kappa_i(\theta)= \epsilon^\prime_i(\theta)=-\theta\psi_i^{\prime\prime}(\theta)>0$, the  heat capacity,
\item
$l(\theta)=\theta[\![\psi^\prime(\theta)]\!]=-\theta[\![\eta(\theta)]\!]$, the latent heat.
\end{itemize}
Furthermore, $d_i(\theta)>0$ denotes the coefficient of heat conduction in Fourier's law,
 $\gamma(\theta)\geq0$ the coefficient of  kinetic undercooling.
As is commonly done, we assume that there exists a unique (constant) {\em melting temperature} $\theta_m$,
characterized by the equation $[\![\psi(\theta_m)]\!]=0.$
Finally, system \eqref{stefan} is to be completed by constitutive equations for
the free energies $\psi_i$ in the bulk phases $\Omega_i(t)$. 

In the sequel we drop the index $i$, as there is no danger of confusion; we just keep in mind that the coefficients depend on the phases.
The temperature is assumed to be continuous across the interface, 
as indicated by the condition $[\![\theta]\!]=0$ in \eqref{stefan}.
However, the free energy and the heat conductivities depend on the respective phases,
and hence the jumps $[\![\psi(\theta)]\!]$,
$[\![\kappa(\theta)]\!]$, $[\![\eta(\theta)]\!]$, $[\![d(\theta)]\!]$ are in general non-zero at the interface.
In this paper we assume that the coefficient of surface tension is constant. 

In this paper, we will prove the following regularity result for $k\in \N\cup \{\infty,\omega\}$.
\begin{theorem}
\label{S1: main theorem: gamma=0}
{\rm ($\gamma \equiv 0$).}
Let $p>m+3$, $\gamma=0$, $\sigma>0$. Suppose that
$d_i\in C^{k+2}(0,\infty)$,  $\psi_i\in C^{k+3}(0,\infty)$ for $i=1,2$ such that
$$\kappa_i(u)=-u\psi_i^{\prime\prime}(u)>0,\quad d_i(u)>0,\quad u\in(0,\infty).$$
Assume the {\em regularity conditions}
$$\theta_0\in W^{2-2/p}_p(\Omega\setminus\Gamma_0)\cap C(\bar{\Omega}),\quad \theta_0>0,
\quad \Gamma_0\in W^{4-3/p}_p,$$
the {\em  compatibility conditions}
$$
\partial_{\nu_\Omega}\theta_0=0,\quad
[\![\psi(\theta_0)]\!]+\sigma {\cH}(\Gamma_0)=0,
\quad [\![d(\theta_0)\partial_{\nu_{\Gamma_0}} \theta_0]\!]\in W^{2-6/p}_p(\Gamma_0),$$
and the {\em well-posedness condition}
$$\quad l(\theta_0)\neq0\quad\mbox{on}\; \Gamma_0.$$
Then there exists a unique $L_p$-solution $(u,\Gamma)$ for the Stefan problem with surface tension \eqref{stefan}
on some possibly small but nontrivial time interval $J=[0,T]$, and
$$
 \mathcal{M}:=\bigcup_{t\in(0,T)}\{\{t\}\times\Gamma(t)\}
$$
is a $C^k$-manifold in $\R^{m+2}$. In particular, each manifold $\Gamma(t)$ is $C^k$ for $t\in (0,T)$.
\end{theorem}
The result in the presence of kinetic undercooling reads as
\begin{theorem}
\label{S1: main theorem: gamma>0}
{\rm ($\gamma > 0$).}
Let $p>m+3$, $\sigma>0$. Suppose that
$d_i, \gamma \in C^{k+2}(0,\infty)$,
$\psi_i\in C^{k+3}(0,\infty)$ for $i=1,2$ such that
$$\kappa_i(u)=-u\psi_i^{\prime\prime}(u)>0,\quad d_i(u)>0,\quad \gamma(u)>0,\quad u\in(0,\infty).$$ Assume
the {\em regularity conditions}
$$\theta_0\in W^{2-2/p}_p(\Omega\setminus\Gamma_0)\cap C(\bar{\Omega}),\quad \theta_0>0,\quad \Gamma_0\in W^{4-3/p}_p,$$
and the {\em compatibility conditions}
\begin{equation*}
\begin{split}
\partial_{\nu_\Omega}\theta_0=0, \;\;
\big([\![\psi(\theta_0)]\!]+\sigma \cH(\Gamma_0)\big)\big(l(\theta_0)-[\![\psi(\theta_0)]\!]-\sigma \cH(\Gamma_0)\big)=\gamma(\theta_0)[\![d(\theta_0)_{\Gamma_0} \theta_0]\!].
\end{split}
\end{equation*}
Then there exists a unique $L_p$-solution $(\theta,\Gamma)$ for the Stefan problem with surface tension \eqref{stefan}
on some possibly small but nontrivial time interval $J=[0,T]$, and
$$ \mathcal{M}:=\bigcup_{t\in(0,T)}\{\{t\}\times\Gamma(t)\}
$$
is a $C^k$-manifold in $\R^{m+2}$. In particular, each manifold $\Gamma(t)$ is $C^k$ for $t\in (0,T)$.
\end{theorem}
\medskip
\begin{remark}
\label{S1: main RMK}
(a) 
For $k\in \N\cup\{\infty\}$, under the conditions in Theorems~\ref{S1: main theorem: gamma=0} and \ref{S1: main theorem: gamma>0}, we can show that the temperature satisfies
$$\theta \in C^k(((0,T)\times\Omega)\setminus \mathcal{M}). $$
See Remark~\ref{S5: Proof of main RMK} for a justification. 
However, in the case $k=\omega$, 
additional work is needed to establish the interior analyticity of $\theta$ 
in the bulk phases. 
In order to keep this already long paper at a reasonable length, 
we will refrain from establishing this here.
\smallskip\\
(b) According to \cite[Theorem 3.9]{PruSimZac12} we obtain a solution of \eqref{stefan} in the state manifold $\mathcal{SM_\gamma}$ on a maximal interval of existence $[0,t_*)$.
The regularity assertions of Theorem~\ref{S1: main theorem: gamma=0}
and Theorem~\ref{S1: main theorem: gamma>0}, respectively, then hold true on $[0,t_*)$.
\qed
\end{remark}

\textbf{Notations:} 
Throughout this paper, we always assume that
\begin{mdframed}
\begin{itemize}
\item $E$ denotes a finite dimensional Banach space.
\item $m+3<p<\infty$, and $s\geq 0$, unless stated otherwise.
\end{itemize}
\end{mdframed}
Given two Banach spaces $X,Y$, the notation $\L(X,Y)$ stands for the set of all bounded linear operators from $X$ to $Y$, and
$\Lis(X,Y)$ denotes the set of all bounded linear isomorphisms from $X$ to $Y$.
\smallskip\\
For any topological sets $U$ and $V$, $U\subset\subset V$ means that $\bar{U}\subset \mathring{V}$ with $\bar{U}$ compact. 
\medskip

\section{\bf Parameterization over a fixed interface}

Let $\Omega\subset\R^{m+1}$ be a bounded domain with boundary $\partial\Omega$ of class $C^2$, and suppose that $\Gamma_0\subset \Omega$ is a closed embedded hypersurface of class $C^2$, that is, a $C^2$-manifold which is the boundary of a bounded domain $\Omega_1(0)\subset \Omega$. Set $\Omega_2(0):=\Omega \setminus \bar{\Omega}_1(0)$. Following \cite{PruSimZac12}, we may approximate $\Gamma_0$ by an $m$-dimensional real analytic compact closed   oriented reference hypersurface $(\M,g)$, with $g$  the tangential metric induced by the Euclidean metric $g_{m+1}$, in the sense that there exists a function $h_0\in C^2(\M,(-a,a))$ such that the map
\begin{align*}
\Lambda_{h_0}:\M\rightarrow \R^{m+1}:\quad [\p\mapsto \p+h_0(\p) \nu_\M(\p)]
\end{align*}
is a diffeomorphism from $\M$ onto $\Gamma_0$, where $\nu_\M$ is the outer normal of $\M$. The positive constant $a$ depends on the inner and outer ball conditions of $\M$. It is well-known that $\M$ admits a $a$-tubular neighborhood $\T_a$, which means that the map
\begin{align*}
\Lambda:\M\times (-a,a)\rightarrow\R^{m+1}:\hspace*{.5em}(\p,r)\mapsto \p+r \nu_\M(\p)
\end{align*}
is a diffeomorphism from the fiber bundle $\M\times (-a,a)$ onto ${\im}(\Lambda):=\T_a$. For sufficiently small $a$, $\bar{\T}_a\subset\Omega$. $\M$ bounds a domain $\Omega_1$, and we set $\Omega_2=\Omega\setminus \bar{\Omega}_1$. The inverse $\Lambda^{-1}$ can be decomposed as 
\begin{align*}
\Lambda^{-1}:\T_a \rightarrow  \M\times(-a,a): \hspace*{.5em} z\mapsto (\Pi(z),d_\M(z)),
\end{align*}
where $\Pi(z)$ is the metric projection of $z$ onto $\M$ and $d_\M(z)$ denotes the signed distance from $z$ to $\M$ such that $|d_\M(z)|={\rm dist}(z,\M)$ and $d_\M(z)<0$ iff $z\in \Omega_1$. It follows from the inverse function theorem that the maps $\Pi$ and $d_\M$ are both real analytic.
\smallskip\\
We may use the map $\Lambda$ to parameterize the unknown free boundary $\Gamma(t)$ over $\M$ by a height function $h(t):\Sigma\to \R$ via
\begin{align*}
\Gamma(t):=\{ \p+h(t,\p)\nu_\M(\p),\quad \p\in\M\},\quad t\geq 0,
\end{align*}
for small $t\geq 0$, at least. $\Gamma(t)$ bounds a bounded region $\Omega_1(t)$. Put $\Omega_2(t)=\Omega\setminus \bar{\Omega}_1(t)$. Choosing an auxiliary function $\zeta\in \mathcal{D}((-2a/ 3,2a/ 3),[0,1])$ such that $\zeta|_{(-a/ 3,a/ 3)}\equiv 1$, we may extend the above diffeomorphism onto $\bar{\Omega}$ via
\begin{align*}
\Xi_h(t,z)=z+\zeta(d_\M(z))h(t,\Pi(z))\nu_\M(\Pi(z))=:z+\Upsilon(h)(t,z).
\end{align*}
We have transformed the time varying regions $\Omega_i(t)$ to the fixed domains $\Omega_i$. This is the direct mapping method, also known as Hanazawa transformation. By means of this transformation, we obtain the following transformed problem for $\vartheta(t,z):=\theta(t,\Xi_h(t,z))$:
\begin{align}
\label{S2: transf Stefan}
\left\{\begin{aligned}
\kappa(\vartheta)\partial_t \vartheta+\mathcal{A}(\vartheta,h)\vartheta&=\kappa(\vartheta)\mathcal{R}(h)\vartheta  &&\text{in}&&\Omega\setminus\M,\\
\partial_{\nu_{\Omega}}\vartheta&=0 &&\text{on}&&\partial\Omega,\\
[\![\vartheta]\!]&=0 &&\text{on}&&\M,\\
[\![\psi(\vartheta)]\!]+\sigma{\cH}(h)&=\gamma(\vartheta)\beta(h)\partial_t h &&\text{on}&& \M, \\
\{l(\vartheta)-\gamma(\vartheta)\beta(h)\partial_t h \}\beta(h)\partial_t h +\mathcal{B}(\vartheta,h)\vartheta&=0 &&\text{on}&& \M,\\
\vartheta(0)=\vartheta_0, \quad
h(0)&=h_0. &&
\end{aligned}\right.
\end{align}
Here $\mathcal{A}(\vartheta,h)$ and $\mathcal{B}(\vartheta,h)$ denote the transformations of $-{\rm div}(d\nabla )$ and
$-[\![d\partial_{\nu_\Gamma}]\!]$, respectively. Moreover, ${\cH}(h) $ stands for the mean curvature of the hypersurface 
$$\Gamma_h := \Lambda_h ( \M).$$
The term $\beta(h)\partial_t h$ represents the normal velocity $V$ with $\beta(h):=(\nu_\M|\nu_\Gamma(h))$,
where $\nu_\Gamma(h)$ is the outer normal field of $\Gamma_h$,
and
$$\mathcal{R}(h)\vartheta :=\partial_t \vartheta -\partial_t \theta\circ \Xi_h.$$
It is shown in \cite{PruSimZac12} that $\nu_\Gamma(h)=\beta(h) (\nu_\M -\alpha(h))$, and
$$\beta(h)=(1+|\alpha(h)|^2)^{-1/2},\hspace{1em} \mathcal{R}(h)\vartheta=(\nabla  \vartheta|{[I + \nabla\Upsilon(h)^{\sf T}]}^{-1}\partial_t\Upsilon(h) ). $$
Here with the Weingarten map $L_\M=-\nabla_\M \nu_\M$. We have
\begin{align*}
\alpha(h):=M_0(h)\nabla_\M h,\quad \text{where}\quad M_0(h)&:=(I- h L_\M)^{-1}.
\end{align*}
With $\partial_\nu:= \partial_{\nu_{\M}}$, the operator $\mathcal{B}(\vartheta,h)$ becomes
\begin{align*}
 \mathcal{B}(\vartheta,h)\vartheta&= -[\![d(\theta)\partial_{\nu_\Gamma} \theta]\!]\circ\Xi_h=-([\![d(\vartheta)(I-M_1(h))\nabla \vartheta]\!]|\nu_\Gamma)\\
 &= -\beta(h)([\![d(\vartheta)(I-M_1(h))\nabla \vartheta]\!]|\nu_\Sigma -\alpha(h))\\
&= -\beta(h)[\![d(\vartheta)\partial_{\nu} \vartheta]\!]
+\beta(h)([\![d(\vartheta)\nabla \vartheta]\!]|(I-M_1(h))^{\sf
T}\alpha(h)),
\end{align*}
where $M_1(h):=[(I+\nabla\Upsilon(h)^{\sf T})^{-1}\nabla\Upsilon(h)^{\sf T}]^{\sf T}$, and finally 
\begin{align*}
 \mathcal{A}(\vartheta,h)\vartheta= & -{\rm div}( d(\theta)\nabla \theta)\circ\Xi_h= -((I-M_1(h))\nabla|d(\vartheta)(I-M_1(h))\nabla \vartheta)\\
 = & -d(\vartheta)\Delta \vartheta + d(\vartheta)[M_1(h)+M_1^{\sf T}(h)-M_1(h)M_1^{\sf T}(h)]:\nabla^2 \vartheta\\
&-d^\prime(\vartheta)|(I-M_1(h))\nabla \vartheta|^2
 + d(\vartheta)((I-M_1(h)):\nabla M_1(h)|\nabla \vartheta).
\end{align*}
We recall that, for matrices  $A,B\in\R^{n\times n}$, $
A:B=\sum_{i,j=1}^n a_{ij}b_{ij}=\text{tr}\,(AB^{\sf T}) $ denotes
the inner product.

We set
\begin{align*}
&\ea:=\{\vartheta\in H^1_p(J;L_p(\Omega))\cap L_p(J; H^2_p(\Omega\setminus\M)): [\![\vartheta]\!]=0, \partial_{\nu_\Omega}\vartheta=0\},\\
&\eb:=
\begin{cases}
W^{3/2-1/2p}_p(J; L_p(\M))\cap W^{1-1/2p}_p(J; H^2_p(\M))\cap L_p(J; W^{4-1/p}_p(\M)),\\
\hspace{27em} \gamma\equiv 0,\\
W^{2-1/2p}_p(J;L_p(\M))\cap L_p(J;W^{4-1/p}_p(\M)),\hspace{9.3em} \gamma>0,
\end{cases}\\
&\ej:=\ea\times\eb,
\end{align*}
that is, $\ej$ denotes the solution space for \eqref{S2: transf Stefan}. 
Similarly, we define
\begin{align}
\label{S2: def of Fj}
\left\{\begin{aligned}
&\fa:=L_p(J;L_p(\Omega)),\\
&\fb:=W^{1-1/2p}_p(J;L_p(\M))\cap L_p(J; W^{2-1/p}_p(\M)),\\
&\fc:=W^{1/2-1/2p}_p(J;L_p(\M))\cap L_p(J; W^{1-1/p}_p(\M)),\\
&\fd:=[W^{2-2/p}_p(\Omega\setminus\M)\cap C(\bar{\Omega})]\times W^{4-3/p}_p(\M).
\end{aligned}\right.
\end{align}
A left subscript zero means vanishing time trace at $t=0$, whenever it exists. So for example
$\zeb=\{h\in \eb:\; h(0)=\partial_t h(0)=0\}$
for $p>3$.
\smallskip\\
Whenever $(J)$ is replaced by $(J,U)$, or $(J,U;E)$  for any set $U$, e.g., in $\ej$, it always means that in the corresponding spaces the original spatial domain is replaced by $U$ and the functions become $E$-valued in the latter case. For example, 
\begin{align*}
\mathbb{F}_2(J,\T_a;E):=W^{1-1/2p}_p(J;L_p(\T_a,E))\cap L_p(J;W^{2-1/p}_p(\T_a,E)).
\end{align*}

We equip $\M$ with a normalized real analytic atlas $\A=(\Ok,\vpk)_{\kappa\in\K}$, where $\K$ is a finite index set,  in the sense that $\vpk(\Ok)=\Q:=(-1,1)^m$, for every $\kappa$. 
Let  $\psk:=[\vpk]^{-1}$.  
Then we may endow $\Omega$ with a real analytic atlas $\Ae=(\Oek,\vpek)_{\kappa\in \Ke}$ with $\Ke:=\K\cup\{\kappa_1,\kappa_2\}$ by the following construction. When $\kappa\in\K$
\begin{align*}
\Oek:=\T_{a,\kappa}:=\Lambda({\Ok\times(-a,a)}),\hspace*{1em} \vpek(z):=(\vpk\circ\Pi(z),d_\M(z)),\hspace*{.5em}z\in\Oek.
\end{align*}
Let $\Qa:=\Q\times(-a,a)$ and 
$$\psek:=[\vpek]^{-1}:\Qa \rightarrow \Oek: \quad[(x,y)\mapsto \psk(x)+y\nu_{\M}(\psk(x))]$$
with $x\in\Q$ and $y\in(-a,a)$. For $\kappa\notin\K$, we set
$$
(\O_{\frak{e},\kappa_i},\varphi_{\frak{e},\kappa_i}):=(\O_{\frak{e},\kappa_i},\id_{\O_{\frak{e},\kappa_i}}),\quad  i=1,2,
$$
where $\O_{\frak{e},\kappa_i}:=\Omega_i\setminus \bar{\T}_{a/6}$ for $i=1,2$. It is immediate from our construction that the atlas $\Ae$ is real analytically compatible with the Euclidean structure of $\Omega$. Moreover, all the transition maps are $C^\infty$-continuous and have bounded derivatives, which we refer to as $BC^\infty$-continuous. 
\smallskip\\
A family $(\pk)_{\kappa\in\K}$ is called a localization system subordinate to $\A$ if 
\begin{center}
$\pk\in\mathcal{D}(\Ok,[0,1])$ and $(\pi_{\kappa}^2)_{\kappa\in\K}$ is a partition of unity subordinate to $(\Ok)_{\kappa\in\K}$.
\end{center}
Based on \cite[Lemma~3.2]{Ama13}, $\M$ admits a localization system $(\pk)_{\kappa\in\K}$ subordinate to $\A$.
Using the cut-off function $\zeta$ introduced above, we may construct a new localization system $(\pek)_{\kappa\in\Ke}$ subordinate to $\Ae$ as follows:
\begin{itemize}
\item[(L1)] For any $\kappa\in\K$, $\pek\in \mathcal{D}(\Oek,[0,1])$ is defined by
\begin{align*}
\pek(z):=
\begin{cases}
\zeta(d_\M(z))\pk(\Pi(z)), \hspace{1em}& z\in \Oek,\\
0, &z\notin \Oek.
\end{cases}
\end{align*}
\item[(L2)] For $i=1,2$, $\pi_{\frak{e},\kappa_i}\in C^\infty(\bar{\O}_{\frak{e}_i,\kappa_i},[0,1])$ is defined by
\begin{align*}
\pi_{\frak{e},\kappa_i}(z)=
\begin{cases}
\displaystyle
\frac{1-\zeta(d_\M(z))}{\sqrt{(1-\zeta(d_\M(z)))^2+\zeta^2(d_\M(z))}}, \hspace{1em}& z\in \Omega_i\cap\T_a,\\
1, &z\in \bar{\Omega}_i\setminus\T_a.
\end{cases}
\end{align*}
\end{itemize}
$(\pi^2_{\frak{e},\kappa})_{\kappa\in\Ke}$ forms a localization system subordinate to $(\Oek)_{\kappa\in\Ke}$. 
We put 
\begin{align*}
\Xk:=
\begin{cases}
\R^{m+1},\hspace*{1em}&\kappa\in\K,\\
\O_{\frak{e},\kappa_i}, &i=1,2.
\end{cases}
\end{align*}
For any finite dimensional Banach space $E$, the maps $\Rcek$ and $\Reek$ are defined by 
\begin{align*}
\Rcek: L_{1,loc}(\Omega,E)\rightarrow L_{1,loc}(\Xk,E):\hspace*{.5em}u \mapsto \kef\pek u, \quad \kappa\in \Ke,
\end{align*}
and
\begin{align*}
\Reek:L_{1,loc}(\Xk,E)\rightarrow  L_{1,loc}(\Omega,E): \hspace*{.5em} u\mapsto \pek\keb u, \quad \kappa\in \Ke.
\end{align*}
Here and in the following it is understood that a partially defined and compactly supported map is automatically extended over the whole base manifold by identifying it to be zero outside its original domain.
Moreover, let
\begin{align*}
\Rce: L_{1,loc}(\Omega,E)\rightarrow \prod\limits_{\kappa\in\Ke}L_{1,loc}(\Xk,E):\hspace*{.5em}u \mapsto (\Rcek u)_{\kappa\in\Ke},
\end{align*}
and
\begin{align*}
\Ree:\prod\limits_{\kappa\in\Ke} L_{1,loc}(\Xk,E)\rightarrow  L_{1,loc}(\Omega,E): \hspace*{.5em} (u_{\kappa})_{\kappa}\mapsto \sum\limits_{{\kappa\in\Ke}}\pek\keb u.
\end{align*}
On the manifold $\M$, similar maps $\Rck$, $\Rc$, $\Rek$, and $\Re$ are defined in terms of $\pk$, $\vpk$, and $\psk$. 
See \cite{ShaoPre}.

\subsection{Function spaces}
For any open subset $U\subset\R^n$, the Banach space $BC^{k}(U,E)$ is defined by
\begin{align*}
BC^{k}(U,E):=(\{u\in C^k(U,E):\|u\|_{k,\infty}<\infty \},\|\cdot\|_{k,\infty}).
\end{align*} 
The closed linear subspace $BU\!C^k(U,E)$ of $BC^{k}(U,E)$ consists of all functions $u\in BC^{k}(U,E)$ such that $\partial^{\alpha}u$ is uniformly continuous for all $|\alpha|\leq k$. Moreover,
\begin{align*}
BC^{\infty}(U,E):=\bigcap_{k}BC^k(U,E)=\bigcap_{k}BU\!C^k(U,E).
\end{align*}
It is a Fr\'echet space equipped with the natural projective topology. 
\smallskip\\
For $0<s<1$, $0<\delta\leq\infty$ and $u\in E^U$, the seminorm $[\cdot]^\delta_{s,\infty}$ is defined by
\begin{align*}
[u]^{\delta}_{s,\infty}:=\sup_{h\in(0,\delta)^n}\frac{\|u(\cdot+h)-u(\cdot)\|_{\infty}}{|h|^s}, \hspace*{1em}[\cdot]_{s,\infty}:=[\cdot]^{\infty}_{s,\infty}.
\end{align*}
Let $k<s<k+1$. The \textbf{H\"older} space $BC^s(U,E)$ is defined as 
\begin{align*}
BC^s(U,E):=(\{u\in BC^k(U,E):\|u\|_{s,\infty}<\infty \},\|\cdot\|_{s,\infty}),
\end{align*}
where $\|u\|_{s,\infty}:=\|u\|_{k,\infty}+\max_{|\alpha|=k}[\partial^{\alpha} u]_{s-k,\infty}$.
\smallskip\\
In order to have a general theory that is applicable to situations other than the Stefan problem,
we also introduce the \textbf{little H\"older} space of order $s\geq 0$, which is defined by
\begin{center}
$bc^s(U,E):=$ the closure of $BC^{\infty}(U,E)$ in $BC^s(U,E)$.
\end{center}
By \cite[formula~(11.13), Corollary~11.2, Theorem~11.3]{AmaAr}, we have 
\begin{align*}
bc^k(U,E)=BU\!C^k(U,E),
\end{align*}
and for $k<s<k+1$
\begin{center}
$u\in BC^s(U,E)$ belongs to $bc^s(U,E)$ iff $\lim\limits_{\delta\rightarrow 0}[\partial^{\alpha}u]^{\delta}_{s-[s],\infty}=0$, \hspace{1em} $|\alpha|=[s]$.
\end{center}
The spaces $\F^s(\M,E)$ with $\F\in \{bc,BC,W_p,H_p\}$ are defined in terms of the smooth atlas $\A$, that is, $u\in\F^s(\M,E)$ iff $\kf u:= u \circ \psk \in \F^s(\Q,E)$ for all $\kappa\in\K$. See \cite{Ama13} and \cite{ShaoPre}.
\smallskip\\
Throughout, for any finite index set $\mathbb{A}$ and Banach spaces $E_\alpha$, $\alpha\in\mathbb{A}$, it is understood that the space $\mathbb{E}:=\prod\limits_{\alpha\in\mathbb{A}}E_\alpha$ is equipped with the maximum norm. 
For $\F\in \{bc,BC,BU\!C,W_p,H_p\}$, we put 
$$\bF^s_{\frak{e}}:=\prod\limits_{\kappa\in\Ke}{\F}^s(\Xk,E),
\quad  \bF^s:=\prod\limits_{\kappa\in\K}{\F}^s_{\kappa},$$ 
where ${\F}^s_{\kappa}=\F^s(\R^m,E)$.
\begin{prop}
\label{S2: retraction}
Suppose that $\frak{B}=\F$ when $\F\in \{bc, W_p,H_p\}$, or $\frak{B}\in \{ BC, BU\!C\}$ when $\F=BC$. Then 
\begin{itemize}
\item[(a)] $\Re$ is a retraction from $\bB^s$ onto $\F^s(\M,E)$ with $\Rc$ as a coretraction. Moreover, the case $s<0$ and $\F\in \{W_p,H_p\}$ is also admissible.
\item[(b)] $\Ree$ is a retraction from $\bB^s_{\frak{e}}$ onto $\F^s(\Omega,E)$ with $\Rce$ as a coretraction.
\end{itemize}
\end{prop}
\begin{proof}
(a) The assertions are special cases of \cite[Theorem~9.3]{AmaAr} and \cite[Propositions~2.1, 2.2]{ShaoPre}.
\smallskip\\
(b) For $\kappa\in\K$, the boundedness of the maps $\Reek$ and $\Rcek$ follows from the proofs in \cite[Section~6, 7]{Ama13} and \cite[Section~12]{AmaAr}. The continuity of the maps $\mathcal{R}^c_{\frak{e},\kappa_i}$ and $\mathcal{R}_{\frak{e},\kappa_i}$ is straightforward.
\end{proof}
Let $J:=[0,T]$ for some $T>0$.
Due to the temporal independence of the above retraction-coretraction systems, we readily infer that
\begin{prop}
\label{S2: retraction-involving time}
Let $r\geq 0$. 
Suppose that $\frak{C}\in\{C,W_p\}$, and  
$\frak{B}=\F$ when $\F\in \{bc, W_p,H_p\}$, or $\frak{B}\in \{ BC, BU\!C\}$ when $\F=BC$.
Then 
\begin{itemize}
\item[(a)] $\Re$ is a retraction from $\frak{C}^r(J;\bB^s)$ onto $\frak{C}^r(J;\F^s(\M,E))$ with $\Rc$ as a coretraction. Moreover, the case $s<0$ and $\F\in \{W_p,H_p\}$ is also admissible.
\item[(b)] $\Ree$ is a retraction from $\frak{C}^r(J;\bB^s_{\frak{e}})$ onto $\frak{C}^r(J;\F^s(\Omega,E))$ with $\Rce$ as a coretraction.
\end{itemize}
\end{prop}
\medskip

\section{\bf Parameter-dependent diffeomorphisms}

\subsection{\bf Model diffeomorphisms in Euclidean space}
Given  $z_c\in \T_{a/3,{\kappa_c}}\setminus\Sigma \subset\Omega$ for some ${\kappa_c}\in \K$, 
let 
$$(x_c,y_c)=(\vpt\circ\Pi(z_c),d_\M(z_c)).$$ 
We may choose a sufficiently small positive constant $\varepsilon_0$ such that 
\begin{align*}
\B^m(x_c,5\varepsilon_0)\subset\Q,\hspace*{1em} \B(y_c,5\varepsilon_0)\subset(-1)^i(0,a/3).
\end{align*}
Let $\B_{\varepsilon,\alpha}:=\B^m(x_c,\varepsilon)\times(-\alpha,\alpha)$. 
We pick several auxiliary functions in the following manner:
\begin{itemize}
\item $\varpi \in\mathcal{D}(-13a/18,13a/18),[0,1])$, and $\varpi|_{(-2a/3,2a/3)}\equiv 1$.
\item $\chi_m\in\mathcal{D}(\B^m(x_c,2\varepsilon_0),[0,1])$, and $\chi_m|_{\B^m(x_c,\varepsilon_0)}\equiv 1$.
\item $\chi\in\mathcal{D}(\B(y_c,2\varepsilon_0),[0,1])$, and $\chi|_{\B(y_c,\varepsilon_0)}\equiv 1$.
\item $\vge\in\mathcal{D}(\pst(\B_{5\varepsilon_0,17a/18}),[0,1])$, and $\vge|_{\pst(\B_{4\varepsilon_0,8a/9})}\equiv 1$. We set $\varsigma:=\vge|_{\M}$
\end{itemize}
Then we can introduce a localized translation on $\Qa$ as follows
\begin{align*}
\theta_{\mu,\eta}:(x,y)\mapsto (x+\chi_m(x)\varpi(y)\mu, y+\chi_m(x)\chi(y)\eta),\hspace*{1em} (\mu,\eta)\in\Br,
\end{align*}
with $\mu\in\R^m$, $\eta\in\R$ and sufficiently small $r_0>0$. Henceforth, $\B(x_0,r_0)$ always denotes the ball with radius $r_0$ centered at $x_0$ in ${\R}^n$. The dimension $n$ of the ball is not distinguished as long as it is clear from the context. 
\smallskip\\
The related partial translations in horizontal/vertical directions can be defined separately as
\begin{align}
\label{S3: sep translations}
\begin{cases}
\theta_\mu: (x,y)\mapsto (x+\chi_m(x)\varpi(y)\mu, y),\\
\bar{\theta}_\eta: (x,y)\mapsto (x, y+\chi_m(x)\chi(y)\eta).
\end{cases}
\end{align}
For $r_0$ small and every $(\mu,\eta)\in\Br$
\begin{align}
\label{S3: comp of tme}
\theta_{\mu,\eta}=\theta_\mu\circ \bar{\theta}_\eta.
\end{align}
It is not hard to show that for sufficiently small $r_0$, we have 
$\theta_{\mu,\eta}(\B_{3\varepsilon_0,7a/9})\subset \B_{3\varepsilon_0,7a/9}$ for any $(\mu,\eta)\in\B(0,r_0)$.
\eqref{S3: sep translations} implies that 
\begin{align}
\label{S3: pull-back&push-forward}
\tme=\te\circ\tm,\quad \text{and}\quad\theta^{\mu,\eta}_\ast:=[\theta_{\mu,\eta}]_\ast=\theta^\mu_\ast\circ\bar{\theta}^\eta_\ast.
\end{align}
Since $\theta_\mu$ and $\bar{\theta}_\eta$ are the restrictions of truncated translations in the sense of those defined in \cite{EscPruSim03} in horizontal and vertical direction, respectively, the readers should have no difficulty convincing themselves that all the results in \cite{EscPruSim03} are at our disposal for functions defined on an open subset $U\subset \R^{m+1}$ and the transformations $\tm$ and $\te$. Taking advantage of these considerations, we will state some properties of $\tme$ for later use in the following with brief proofs. 

Hereafter, we assume $O\subset U\subset\R^{m+1}$ are open, and $O$ contains $\bar{\B}_{3\varepsilon_0,7a/9}$.
We conclude from $\theta_{\mu,\eta}(U)\subset U$ that $\tme:\F^s(U,E)\rightarrow E^{U}$. 
In the rest of this section, we always assume the following, unless mentioned otherwise, 
\begin{mdframed}
$$1<p<\infty,\quad n,l\in\Nz,\quad k\in\Nz\cup\{\infty,\omega\}.$$
\end{mdframed}
\begin{prop}
\label{S3: U-continuity}
Suppose that $\F\in\{bc, BU\!C, W_p, H_p\}$ and $u\in C^{n+k}(O,E)\cap \F^s(U,E)$ with $s\in [0,n]$ when $\F\in\{bc, W_p, H_p\}$, or with $s=n$ when $\F=BU\!C$. Then for sufficiently small $r_0$,
\begin{itemize}
\item[(a)] $\theta_{\mu,\eta}\in{\Dfi}(U)$, $(\mu,\eta)\in\Br.$
\item[(b)] 
$\tme\in \Lis(\F^s(U,E))$, $[\tme]^{-1}=\theta^{\mu,\eta}_\ast$, $(\mu,\eta)\in\Br.
$
Moreover, there exists a constant $M>0$ such that
$$\|\tme\|_{\L(\F^s(U,E))}\leq M,\hspace*{1em} (\mu,\eta)\in\Br.$$
\item[(c)] $[(\mu,\eta)\mapsto \tme u]\in C^k(\Br,\F^s(U,E)).$
\end{itemize}
\end{prop}
\begin{proof}
(a) The statement is immediate from \eqref{S3: pull-back&push-forward} and \cite[Proposition~2.2]{EscPruSim03}.
\smallskip\\
(b) We infer from (a) that 
$$\tme\circ\tmei=\tmei \circ \tme =\id_{\F^s(U,E)}.$$ 
\cite[Proposition~2.4]{EscPruSim03} implies that for $r$ small enough
$$\theta_\mu^*, \bar{\theta}_\eta^*\in \Lis(\F^s(U,E)),\quad (\mu,\eta)\in\Br.$$
Moreover, there exists a constant $M>0$ such that
$$\|\theta_\mu^*\|_{\L(\F^s(U,E))} + \|\bar{\theta}_\eta^*\|_{\L(\F^s(U,E))} \leq M,\quad (\mu,\eta)\in\Br.$$
The assertion now follows from \eqref{S3: pull-back&push-forward}. 
\smallskip\\
(c)  It is a simple matter to see that the assertion holds for the transformation $\tm$ and $\te$, respectively, namely that we have
\begin{align}
\label{S3: C^k-1}
[\mu\mapsto \tm u]\in C^k(\Br,\F^s(U,E)),\hspace{1em} [\eta\mapsto \te u]\in C^k(\Br,\F^s(U,E))
\end{align}
with $\partial^{\alpha}_{\mu}[\tm u]=\chi_m^{|\alpha|}\varpi^{|\alpha|}[\tm\partial^{\alpha}_{x}u]$, and $\partial^{\beta}_{\eta}[\te u]=\chi_m^{\beta}\chi^\beta[\te\partial^{\beta}_{y}u]$. Then we can obtain from \eqref{S3: pull-back&push-forward}, \eqref{S3: C^k-1}, the point-wise multiplier theorem in \cite[Section~3.3.2]{Trie78}, and induction that the map
\begin{align*}
(\mu,\eta)\mapsto\partial^{(\alpha,\beta)}_{(\mu,\eta)}[\tme u]=\chi_m^{|\alpha|+\beta}\chi^\beta\varpi^{|\alpha|}[\tme\partial^{(\alpha,\beta)}_{(x,y)}u],\hspace*{1em} |\alpha|+\beta\leq k,
\end{align*}
is separately $C^{k-|\alpha|-\beta}$-continuous into the space $\F^s(U,E)$. So the case $k\in\Nz\cup\{\infty\}$ follows immediately from a well-known fact in multi-variable calculus.
When $k=\omega$, the proof is basically the same as that of \cite[Proposition~3.2]{EscPruSim03}.
\end{proof}
Let
$$
\mathcal{A}:=\sum\limits_{|\alpha|\leq l}a_\alpha \partial^\alpha,
$$ 
be an $l$-th order linear  differential operator on $U$
with coefficients $a_\alpha:U\rightarrow \L(E)$, where $\alpha\in\Nz^{m+1}$. We set $\tAme:=\tme\mathcal{A}\theta^{\mu,\eta}_\ast$.
\begin{theorem}
\label{S3: U-diff op}
Suppose that $a_\alpha\in C^{n+k}(O,\L(E))\cap BU\!C^n(U,\L(E))$. Then
\begin{align*}
[(\mu,\eta)\mapsto \tAme]\in C^k(\Br, \L(\F^{s+l}(U,E),\F^s(U,E))),
\end{align*}
where $\F\in\{bc,BU\!C,W_p,H_p\}$, and $s\in [0,n]$.
\end{theorem}
\begin{proof}
By modifying the proof of \cite[Proposition~4.1]{EscPruSim03} in an obvious way, one may verify the case of constant coefficients. 
Now the assertion is a direct consequence of  point-wise multiplication results, and Proposition~\ref{S3: U-continuity}(c).
\end{proof}

\subsection{\bf Parameter-dependent diffeomorphisms on $\Omega$ and $\Omega_i$}

By means of $\theta_{\mu,\eta}$, we are now in a position to introduce a family of localized translations on $\Omega$ and $\Omega_i$, respectively, by
\begin{align*}
\Theta_{\mu,\eta}(z)=
\begin{cases}
\pset\circ\theta_{\mu,\eta}\circ\vpet(z),\hspace*{1em} &z\in\Oet,\\
z, &z\in \Omega\setminus\Oet.
\end{cases}
\end{align*}
Thanks to Proposition~\ref{S3: U-continuity}(a), $\Theta_{\mu,\eta}\in{\Dfi}(\Omega)\cap {\Dfi}(\Omega_i)$ is evident. 
\smallskip\\
There is a universal extension operator $\mathscr{E}_i\in \L(\F^s(\Omega_i,E), \F^s(\Omega,E))$.
The restriction operator from $\F^s(\Omega,E)$ to $\F^s(\Omega_i,E)$ is denoted by $\mathscr{R}_i$. Then $\mathscr{R}_i$ is a retraction from $\F^s(\Omega,E)$ onto $\F^s(\Omega_i,E)$ with $\mathscr{E}_i$ as a coretraction. When $u\in \F^s(\Omega\setminus\Sigma,E)$, it is understood that $\mathscr{E}_i u:= \mathscr{E}_i (u|_{\Omega_i})$.
\smallskip\\
For $u\in E^{\Omega}$, its pull-back and push-forward by $\Theta_{\mu,\eta}$ can be explicitly expressed as
$$
{\Tme}u={\kbet}{\tme}{\kfet}(\vge u)+(1_{\Omega}-{\vge})u, \quad
{\Tmei}u={\kbet}{\tmei}{\kfet}(\vge u)+(1_{\Omega}-{\vge})u.$$
Analogously, given  $u\in E^{\Omega_i}$ of $u$, we have
\begin{align}
\label{S3: Omega-Omega_i}
{\Tme}u=\mathscr{R}_i\circ{\Tme}\circ \mathscr{E}_i u, \quad {\Tmei}u=\mathscr{R}_i\circ{\Tmei}\circ \mathscr{E}_i u.
\end{align}
The construction of $\Theta_{\mu,\eta}$ implies that $\Tme(E^{\Omega_i})\subset E^{\Omega_i}$, and likewise for $\Tmei$. Suppose that $\F\in\{bc,BC, W_p,H_p\}$. We define 
\begin{align*}
\F^{s,\Omega}_\cp:=\{u\in \F^s(\Omega,E): \supp(u)\subset \pset(\bar{\B}_{5\varepsilon_0,17a/18})\},
\end{align*}
and
\begin{align*}
\F^{s}_\cp:=\{u\in \F^s(\R^{m+1},E): \supp(u)\subset \bar{\B}_{5\varepsilon_0,17a/18}\}.
\end{align*}
The analogue of \cite[Lemma~3.1]{ShaoPre} is at our disposal, i.e., it holds that
\begin{lem}
\label{S3: main lem}
$\kbet\in \Lis(\F^s_\cp, \F^{s,\Omega}_\cp)$ with $[\kbet]^{-1}=\kfet$.
\end{lem}
\begin{proof}
Since \cite[(3.2)]{ShaoPre} still holds with $\nabla$ and $g$ denoting the usual gradient and Euclidean metric in $\R^{m+1}$, respectively,
the proof is essentially the same as that of \cite[Lemma~3.1]{ShaoPre}.
\end{proof}
By means of Lemma~\ref{S3: main lem}, with only minor changes to the proofs in \cite[Section~3]{ShaoPre}, one can show that $\Theta_{\mu,\eta}$ inherits all the properties of its counterpart therein. 
\begin{prop}
\label{S3: Omega-continuity}
Suppose that $\F\in\{bc,BC, W_p,H_p\}$. Then 
\begin{align*}
\Tme\in \Lis(\F^s(\Omega,E))\cap \Lis(\F^s(\Omega_i,E)),\hspace*{1em}[\Tme]^{-1}=\Tmei,\hspace*{1em} (\mu,\eta)\in\Br.
\end{align*}
Moreover, there exists a constant $M>0$ such that
\begin{align*}
\|\Tme\|_{\L(\F^s(\Omega,E))} + \|\Tme\|_{\L(\F^s(\Omega_i,E))}\leq M,\hspace*{1em} (\mu,\eta)\in\Br.
\end{align*}
\end{prop}
\begin{proof}
The proof is almost the same as that of \cite[Proposition~3.3]{ShaoPre}. For the reader's convenience, we will present a proof  based on Lemma~\ref{S3: main lem} for the case of $\Omega_i$, and the other case follows in a similar manner. The identity 
$$\Tme\circ\Tmei=\Tmei\circ\Tme=\id_{\F^s(\Omega,E)}\cap\id_{\F^s(\Omega_i,E)}$$ 
is obvious from the definitions of $\Tme$ and $\Tmei$ and $\theta_{\mu,\eta}$. In the formula 
$$ {\Tme}u=\kbet{\tme}\kfet(\vge u)+(1_{\Omega}-{\vge})u ,$$
it can easily verified that the multiplication operators $\vge$ and $(1_{\Omega}-{\vge})$ are uniformly bounded, i.e.,
$$\|\vge\|_{\L(\F^s(\Omega,E))}+ \|(1_{\Omega}-{\vge})\|_{\L(\F^s(\Omega,E))}\leq C_1 $$
for some constant $C_1$. Lemma~\ref{S3: main lem} implies that for some constant $C_2$
$$\|\kbet\|_{\L(\F^{s}_{\cp}, \F^{s,\Omega}_\cp)}+\|\kfet\|_{\L(\F^{s,\Omega}_\cp, \F^{s}_{\cp})} \leq C_2.$$
It follows from Proposition~\ref{S3: U-continuity}(b) that there is a uniform bound $C_3$ such that
$$\|\tme\|_{\L(\F^{s}_{\cp})}\leq C_3. $$
The statement then is a direct consequence of the open mapping theorem and \eqref{S3: Omega-Omega_i}.
\end{proof}

\subsection{\bf Time-dependence}
Let $J=[0,T]$, $T>0$. Assume that $I\subset\mathring J$ is an open interval and $t_c\in I$ is a fixed point. Choose $\varepsilon_{0}$ so small that $\mathbb{B}(t_c,3\varepsilon_0)\subset\subset I$. Pick an auxiliary function 
\begin{align*}
\xi\in\mathcal{D}(\mathbb{B}(t_c,2\varepsilon_{0}),[0,1])\hspace{1em}\text{ with }\hspace{.5em} \xi|_{\mathbb{B}(t_c,\varepsilon_{0})}\equiv{1}. 
\end{align*}
The localized temporal translation is defined by
\begin{align*}
\varrho_{\lambda}(t):=t+\xi(t)\lambda,\hspace{1em} t\in J \text{ and } \lambda\in{\R}.
\end{align*}
For $v:J\times\Qa\rightarrow E$, the parameter-dependent diffeomorphism can be expressed as
\begin{align}
\label{S3: Vlme}
\tilde{v}_{\lambda,\mu,\eta}(t,\cdot):=\tlme v(t,\cdot)=\tTue(t)\rl v(t,\cdot), \hspace{1em} (\lambda,\mu,\eta)\in\Br,
\end{align}
where $\tTue(t):=\theta^{\ast}_{\xi(t)\mu,\xi(t)\eta}$ for $t\in J$.
\smallskip\\
Given $v:J\times\Q\rightarrow E$, analogously, setting $\tTu(t):=\theta^{\ast}_{\xi(t)\mu}$, we can define
\begin{align}
\label{S3: Vlm}
\tilde{v}_{\lambda,\mu}(t,\cdot):=\tlm v(t,\cdot)=\tTu(t)\rl v(t,\cdot), \hspace{1em} (\lambda,\mu)\in\Br.
\end{align}
It is understood that $\theta_\mu$ in \eqref{S3: Vlm} is restricted on the hyperplane $y=0$, appearing as a special case of \eqref{S3: sep translations}.
\smallskip\\
For any $u\in J\times\Omega\rightarrow E$, or $u\in J\times\Omega_i\rightarrow E$, we define
$$u_{\lambda,\mu,\eta}(t,\cdot):=\Tlme u(t,\cdot):=\Tue(t)\rl u(t,\cdot), \hspace*{1em} (\lambda,\mu,\eta)\in\Br.$$
Here $\Tue(t):=\Theta^*_{\xi(t)\mu,\xi(t)\eta}$ for $t\in J$. By Proposition~\ref{S3: Omega-continuity}, for every $t$, $\Tue(t)$ is invertible. 
\smallskip\\
Let $\Tlmei:=[\Tlme]^{-1}$. Note that $u_{\lambda,\mu,\eta}(0,\cdot)=u(0,\cdot)$.

The following lemma strengthens \cite[Lemma~5.1]{EscPruSim03}.
\begin{lem}
\label{S3: involve time}
Suppose that $[(\mu,\eta)\mapsto f(\mu,\eta)]\in C^k+j(\Br,X)$ for $j=0,1$ and some Banach space $X$. Let $F(\mu,\eta)(t):=f(\xi(t)(\mu,\eta))$ with $t\in J$. Then
\begin{align*}
[(\mu,\eta)\mapsto F(\mu,\eta)]\in C^k(\Br,C^j(J,X)).
\end{align*}
\end{lem}
\begin{proof}
To economize the notations, we let $\upsilon:=(\mu,\eta)$. The case $j=0$ is proved in \cite[Lemma~5.1]{EscPruSim03}. We will only treat the case $k=\omega$ and $j=1$. The remaining cases follow in a similar way.
\smallskip\\
Given $\upsilon_0\in \Br$, for every $\upsilon\in \B(\upsilon_0,r)$ with $r$ small enough
$$f(\upsilon)=\sum\limits_\beta \frac{1}{\beta !} \partial^\beta f(\upsilon_0) (\upsilon-\upsilon_0)^\beta,\quad \beta\in\Nz^{m+1},
$$
and there exist constants $M,R$ such that 
$$\|\partial^\beta  f(\upsilon)\| \leq M \frac{\beta !}{R^{|\beta|}},\hspace*{1em} \upsilon\in\B(\upsilon_0,r),\quad  \beta\in \Nz^{m+1}.$$
The constants $r,M,R$ depend continuously on $\upsilon_0$. 
Thus we can find uniform constants $r,M,R$ for $\bar{\B}^{m+1}(0,|\upsilon_0|)$. For each $t\in J$ and $\upsilon\in\B(\upsilon_0,r)$, the following series
\begin{align*}
\sum\limits_\beta \frac{\xi^{|\beta|}(t)}{\beta !} \partial^\beta  f(\xi(t)\upsilon_0) (\upsilon-\upsilon_0)^\beta
\end{align*}
converges in $X$ and represents $F(\upsilon)(t)$. The temporal derivatives can be computed as follows. For $\beta\neq 0$,
\begin{align*}
&\quad \|\frac{d}{dt}[\xi^{|\beta|}(t)\partial^\beta  f(\xi(t)\upsilon_0)]\|_{C(J,X)}\\
& \leq \| \xi^\prime(t) |\beta|\{\xi^{|\beta|-1}(t)\partial^\beta  f(\xi(t)\upsilon_0) +\xi^{|\beta|}(t)\sum\limits_{j=1}^{m+1} \partial^{\beta+e_j} f(\xi(t)\upsilon_0)(\upsilon_0)_j \} \|_{C(J,X)}\\
&\leq M\|\xi \|_{1,\infty} (|\beta|\frac{\beta!}{R^{|\beta|}} + r_0  \sum\limits_{j=1}^{m+1} \frac{(\beta+e_j)!}{R^{|\beta|+1}})\leq M^\prime \frac{\beta !}{(R/C)^{|\beta|}}
\end{align*}
for some sufficiently large $C$. Now the assertion follows right away.
\end{proof}

Suppose that $\Oe \subset\R^{m+1}$ is an open subset containing $\pset(\bar{\B}_{3\varepsilon_0,7a/9})$. Put $\Oei:=\Oe\cap \Omega_i$.
\smallskip\\
When $k\in\Nz\cup\{\infty\}$, we say $u\in C^k(\Oei,E)\cap\F^s(\Omega_i,E)$ if $\mathscr{E}_i u\in C^{k}(\Oe,E)\cap\F^s(\Omega,E)$. 
By convention,  $u\in C^\omega(\Oei,E)\cap\F^s(\Omega_i,E)$ means that $u$ has an analytic extension $u^e\in C^{\omega}(\Oe,E)\cap\F^s(\Omega,E)$. 
\begin{prop}
\label{S3: Omega-Tue}
Suppose that $(S,V)\in\{(\Oe,\Omega), (\Oei,\Omega_i) \}$. 
\begin{itemize}
\item[(a)] Suppose that $u\in C^{n+k+j}(S,E)\cap{\F}^{s}(V,E)$ for $j=0,1$, where either $s\in [0,n]$ if $\F \in\{bc,W_p,H_p\}$, or $s=n$ if $\F=BC$. Here $k\in \Nz\cup\{\infty\}$ if $(S,V)= (\Oei,\Omega_i)$. Then we have
$$[(\mu,\eta)\mapsto \Tue u]\in C^k(\Br,C^j(J;\F^s(V,E))).$$
\item[(b)] Suppose that $\mathcal{A}=\sum\limits_{|\alpha|\leq l} a_\alpha\partial^\alpha$ with $a_\alpha\in C^{n+k+j}(S,\L(E))\cap BC^n(V,\L(E))$  for $j=0,1$. Then 
$$[(\mu,\eta)\mapsto \Tue\mathcal{A}\Tuei]\in C^k(\Br,C^j(J;\mathcal{L}(\F^{s+l}(V,E),\F^{s}(V,E)))),$$
for $s\in[0,n]$ if ${\F}\in\{BC,W_{p},H_{p}\}$, or $s\in[0,n)$ if $\F=bc$.
\end{itemize}
\end{prop}
\begin{proof}
(a) Modifying the proof of \cite[Theorem~3.4]{ShaoPre} as in that of Proposition~\ref{S3: Omega-continuity}, one can show by means of Proposition~\ref{S3: U-continuity}(c) that
$$[(\mu,\eta)\mapsto \Tme u]\in C^k(\Br,\F^s(V,E)). $$
Set $X=\F^s(V,E)$, $f(\mu,\eta)=\Tme u$. Now (a) is a direct consequence of Lemma~\ref{S3: involve time}.
\smallskip\\
(b) When $(S,V)=(\Oei,\Omega_i)$, let $\mathcal{A}_i:=\sum\limits_{|\alpha|\leq l}( a_\alpha^e)\partial^\alpha$ on $\Omega$. Then 
$$\Ame= \mathscr{R}_i \circ \Tme\mathcal{A}_i \Theta^{\mu,\eta}_\ast \circ \mathscr{E}_i.$$ 
The atlas $\Ae$ is real analytically compatible with the Euclidean structure. Similar to the proof of \cite[Proposition~3.6]{ShaoPre}, by Theorem~\ref{S3: U-diff op}, we can show that
$$[(\mu,\eta)\mapsto \Ame:=\Tme\mathcal{A}\Tmei]\in C^k(\Br, \L(\F^{s+l}(V,E),\F^s(V,E))).$$
Set $X=\L(\F^{s+l}(V,E),\F^s(V,E))$, and $f(\mu,\eta)=\Ame$. The assertion follows by Lemma~\ref{S3: involve time}.
\end{proof}
\begin{remark}
\label{S3: rl-differentiability}
Given any Banach space $X$, suppose that $u\in C^{n+k}(I,X)\cap{\F}^s(J,X)$, where either $s\in [0,n]$ if $\F \in\{bc,W_p,H_p\}$, or $s=n$ if $\F=BC$. 
Following the proofs in \cite[Section~3]{EscPruSim03}, we can verify that 
$$[\lambda\mapsto \rl u]\in C^k(\Br,\F^s(J,X)). 
$$
\qed 
\end{remark}

The localized translation $\Theta_{\mu,\eta}$, being restricted to $\M$, induces a family of diffeomorphism $\{\Tm: \mu\in\Br\}$ on function spaces over $\M$, which was introduced in \cite{ShaoPre}. For the reader's convenience, we will briefly state some of its properties herein. 
\begin{align*}
\Theta_\mu(q)=
\begin{cases}
\pst(\theta_\mu(\vpt(q))), \hspace{1em}&q\in \Ot,\\
q, &q\in \M\setminus \Ot.
\end{cases}
\end{align*}
It is evident that $\Theta_\mu \in{\Dfi}(\M)$ for $\mu\in\Br$ with sufficiently small $r_0$. Given $u \in E^\M$, $\Tm$ and $\Tmi$ can be explicitly expressed as  
\begin{align*}
\Tm u=\kbt \tm \kft(\varsigma u)+(1_\M-{\varsigma})u, \hspace{1em}\Tmi u=\kbt \tmi \kft(\varsigma u)+(1_\M-{\varsigma})u,
\end{align*}
respectively. Given $u\in E^{J\times\M}$, setting $\Tu(t):=\Theta^*_{\xi(t)\mu}$, we define
\begin{align*}
u_{\lambda,\mu}(t,\cdot):=\Tlm u(t,\cdot)=\Tu(t)\rl u(t,\cdot),\hspace{1em} (\lambda,\mu)\in\Br.
\end{align*}
Let $\Tlmi:=[\Tlm]^{-1}$. 
\smallskip\\
A linear operator $\mathcal{A}:C^\infty(\M,E)\rightarrow E^\M$ is called a linear differential operator of order $l$ on $\M$ if in every local chart $(\Ok,\vpk)$ with $\kappa\in\K$, there exists some linear differential operator defined on $\Q$
\begin{align*}
\mathcal{A}_\kappa(x,\partial):=\sum\limits_{|\alpha|\leq l}a^\kappa_\alpha(x)\partial^\alpha,\hspace*{.5em}\text{ with }a^\kappa_\alpha\in \L(E)^{\Q},
\end{align*}
called the local representation of $\mathcal{A}$ in $(\Ok,\vpk)$, such that for any $u\in C^\infty(\M,E)$
\begin{align}
\label{S3:local exp}
\kf(\mathcal{A}u)=\mathcal{A}_{\kappa}(\kf u), \hspace*{1em}\kappa\in\K.
\end{align}
\begin{prop}
\label{S3：M-reg}
Let $\O:=\Oe\cap\M$.
\begin{itemize}
\item[(a)] Suppose that $u\in C^{n+k+j}(\M,E)\cap \F^s(\M,E)$  for $j=0,1$, where either $s\in[0,n]$ if $\F\in\{bc,W_p,H_p\}$, or $s=n$ if $\F=BC$. Then we have
$$[\mu\mapsto\Tu u]\in C^k(\Br,C^j(J;\F^s(\M,E))).$$
\item[(b)] 
Suppose that $\mathcal{A}$ is a linear differential operator on $\M$ of order $l$ satisfying that for all $|\alpha|\leq l$ and $\kappa\in\K$, $a_\alpha^\kappa \in BC^n(\Q,\L(E))$ and $a_\alpha^{\kappa_c} \in C^{n+k+j}(\vpt(\O),\L(E))$  for $j=0,1$. Then for $s\in [0,n]$ if $\F=BC$, or $s\in[0,n)$ if $\F=bc$, or $s\in (-\infty,n]$ if $\F\in\{W_p,H_p\}$ 
$$[\mu \mapsto \Tu\mathcal{A}\Tui]\in C^k(\Br,C^j(J;\mathcal{L}(\F^{s+l}(\M,E),\F^s(\M,E)))).$$
\end{itemize}
\end{prop}
\begin{proof}
The proofs for (a) and the case $s\geq 0$ in (b) are given in \cite[Section~3]{ShaoPre}. We will only treat the case $k=\omega$ when $s<0$, the other cases follow similarly. Firstly, we show that for any $s\in (-\infty,0)$ and $\F_p\in \{W_p,H_p\}$
\begin{align*}
[\mu \mapsto \tm \partial_j \tmi]\in C^\omega(\Br, \L(\F^{s+1}_p(\R^m,E),\F^s_p(\R^m,E))).
\end{align*}
(i) $s\leq -1$. 
On account that with $p^\prime$ denoting the H\"older duality of $p$, 
\begin{align*}
\partial_j\in \L(\F^{s+1}_p(\R^m,E),\F^s_p(\R^m,E)), \hspace{.5em}\text{ and }\hspace{.5em}\F^s_p(\R^m,E)=(\F^{-s}_{p^\prime}(\R^m,E))^\prime ,
\end{align*}
for every $u\in \F^{s+1}_p(\R^m,E)$ and $v\in \F^{-s}_{p^\prime}(\R^m,E)$, we have
\begin{align*}
&\quad\langle \tm \partial_j \tmi u, v \rangle:=- \langle u, \tm\partial_j[\tmi v |\det(D(\theta_\mu)^{-1})|]|\det(D\theta_\mu)|\rangle\\
&=-\langle u, \tm\partial_j\tmi v \tm(\det(D(\theta_\mu)^{-1}))\det(D\theta_\mu) \rangle
 -\langle u, v \tm\partial_j [\det(D(\theta_\mu)^{-1})] \det(D\theta_\mu)\rangle.
\end{align*}
Here $\langle \cdot,\cdot \rangle$ denotes the duality pairing from $\F^s_p(\R^m,E)\times\F^{-s}_{p^\prime}(\R^m,E)$ to $\R$. 
By \cite[Proposition~4.1]{EscPruSim03}, we have
\begin{align*}
[\mu\mapsto \tm\partial_j\tmi]\in C^\omega (\Br, \L(\F^{-s}_{p^\prime}(\R^m,E),\F^{-s-1}_{p^\prime}(\R^m,E))).
\end{align*}
It is a simple matter to check that $\tm(\det(D(\theta_\mu)^{-1}))\det(D\theta_\mu)=1$. One may compute 
\begin{align*}
&\quad-\tm\partial_j [\det(D(\theta_\mu)^{-1})] \det(D\theta_\mu)\\
&=-\tm[\partial_j [\det(D(\theta_\mu)^{-1})] \tmi\det(D\theta_\mu)]=\tm\partial_j\tmi \det(D\theta_\mu) \tm\det(D(\theta_\mu)^{-1}).
\end{align*}
We immediately have for all $k\in\Nz$ that
\begin{align*}
[\mu\mapsto \det(D\theta_\mu)]\in C^\omega(\Br, BC^{k+1}(\R^m)).
\end{align*}
It follows again from \cite[Proposition~4.1]{EscPruSim03} that
\begin{align*}
[\mu\mapsto \tm\partial_j\tmi \det(D\theta_\mu)]\in C^\omega(\Br, BC^k(\R^m)).
\end{align*}
On the other hand, we obtain
\begin{align*}
[\mu\mapsto \tm \det(D(\theta_\mu)^{-1})]=1/\det(D\theta_\mu)\in C^\omega(\Br, BC^k(\R^m)).
\end{align*}
Therefore, for every $u\in \F^{s+1}_p(\R^m,E)$ and $v\in  \F^{-s}_{p^\prime}(\R^m,E)$, 
\begin{align*}
[\mu\mapsto \langle \tm \partial_j \tmi u, v \rangle] \in C^\omega(\Br).
\end{align*}
Now \cite[Proposition~1]{Browd62} implies that
\begin{align*}
[\mu\mapsto \tm\partial_j\tmi ] \in C^\omega(\Br, \L(\F^{s+1}_p(\R^m,E), \F^s_p(\R^m,E))).
\end{align*}
(ii) $s\in (-1,0)$. \cite[Proposition~4.1]{EscPruSim03} and (i) show that
\begin{align*}
[\mu\mapsto \tm\partial_j\tmi ] \in & C^\omega(\Br, \L(\F^{s+2}_p(\R^m,E), \F^{s+1}_p(\R^m,E)))\\
&\cap C^\omega(\Br, \L(\F^s_p(\R^m,E), \F^{s-1}_p(\R^m,E))).
\end{align*}
Thus for any $\mu_0\in \Br$, there exists some constants $r_i$, $M_i$ and $R_i$ with $i=1,2$ such that for all $\mu\in \B^m(\mu_0,r_i)$ 
\begin{align}
\label{S3: Cauchy est}
\| \frac{\partial^\alpha}{\partial \mu^\alpha} [\tm\partial_j\tmi]\|_{X_i} \leq M_i\frac{\alpha !}{R_i^{|\alpha|}},\hspace{1em} \alpha\in \N_0^m.
\end{align}
Here $X_1:=\L(\F^{s+2}_p, \F^{s+1}_p)$ and $X_2:=\L(\F^s_p, \F^{s-1}_p)$. 
It follows from interpolation theory that $[\mu\mapsto \tm\partial_j\tmi ] \in  C^\infty(\Br, \L(\F^{s+1}_p(\R^m,E), \F^s_p(\R^m,E)))$. Indeed, we have for $h=1,2,\cdots,m$ and any $\mu\in \Br$
$$ \| \frac{\theta^*_{\mu+t e_h} \partial_j \theta_*^{\mu+t e_h} -\tm\partial_j\tmi}{t} -\frac{\partial}{\partial \mu_h} \tm\partial_j\tmi \|_{X_i}  \to 0, \quad t\to 0,\quad i=1,2.$$
By interpolation theory, the above limits converge when $X_i$ is replaced by $X:=\L(\F^{s+1}_p(\R^m,E), \F^s_p(\R^m,E))$. Continuity, or even $C^\infty$-smoothness, of $\tm\partial_j\tmi$ in $X$ can be verified analogously.
\smallskip\\
Similar Cauchy estimate to \eqref{S3: Cauchy est} holds for $X$ by interpolation theory, i.e., for any $\mu_0\in \Br$, there exists some constants $r=\min r_i$, $M=\max M_i$ and $R=\min R_i$ such that for all $\mu\in \B^m(\mu_0,r)$ 
\begin{align*}
\|\frac{\partial^\alpha}{\partial \mu^\alpha} [\tm\partial_j\tmi]\|_{\L(\F^{s+1}_p, \F^s_p)} \leq M\frac{\alpha !}{R^{|\alpha|}},\hspace{1em} \alpha\in \N_0^m.
\end{align*}
It is well-known that this estimate implies that 
\begin{align*}
[\mu\mapsto \tm\partial_j\tmi ] \in C^\omega(\Br, \L(\F^{s+1}_p(\R^m,E), \F^s_p(\R^m,E))).
\end{align*}
(iii) For $s\in (-\infty,n]$, in view of the proofs for \cite[Proposition~4.1, Theorem~4.2]{EscPruSim03}, we thus infer that for any linear differential operator $\tilde{\mathcal{A}}=\sum_{|\alpha|\leq l}a_\alpha \partial^\alpha$, if $a_\alpha\in BC^n(\R^m,\L(E))\cap C^{n+k+j}(O,\L(E))$ for $O:=\varphi_{\kappa_c}(\O)$,
then 
\begin{align*}
[\mu\mapsto \tm\tilde{\mathcal{A}}\tmi] \in C^\omega (\Br, \L(\F^{s+l}_p(\R^m,E), \F^s_p(\R^m,E))).
\end{align*}
A similar proof to (i) shows that \cite[Lemma~3.1]{ShaoPre} still holds for Sobolev-Slobodeckii and Bessel potential spaces of negative order. Indeed, let
$$\F^{s,\M}_\cp:=\{u\in \F^s(\M,E): {\supp}(u)\subset\pst(\bar{\B}(x_c,5\varepsilon_0)) \}, $$
and
$$\F^{s,\R^m}_\cp:=\{u\in \F^s(\R^m,E): {\supp}(u)\subset \bar{\B}(x_c,5\varepsilon_0) \} . $$
Pick $\bar{\B}(x_c,5\varepsilon_0)\subset\mathring{U}\subset\subset\Q$ with $U$ closed, and $\tau\in \mathcal{D}(\mathring{U},[0,1])$ with $\tau|_{\bar{\B}(x_c,5\varepsilon_0)}\equiv 1$. 
For $r> 0$, given any $u\in W^r_{p,\cp}$ and $v\in W^{-r}_{p^\prime}(\R^m,E)$, we have
\begin{align*}
|\langle  u, \kbt v\rangle_\M |&= \|\langle \kft u, \tau v \sqrt{|\det G|}\rangle|\\
&\leq M \|\kft u\|_{W^r_p(\R^m)} \|v\|_{W^{-r}_{p^\prime}(\R^m)}
\leq M \|u\|_{W^r_p(\M)} \|v\|_{W^{-r}_{p^\prime}(\R^m)}.
\end{align*}
Here $G$ is the local matrix expression of the metric $g$, and
$\langle \cdot,\cdot \rangle_\M$ denotes the duality pairing from $\F^r_p(\M,E)\times\F^{-r}_{p^\prime}(\M,E)$ to $\R$. 
The ultimate line follows from the point-wise multiplier theorem, see \cite[Section~9]{Ama13}, and a similar assertion to \cite[Lemma~3.1]{ShaoPre}. 
Thus the constant $M$ is independent of $u$ and $v$. It implies that for $r\in\R$
$$\kbt\in \Lis(\F^{r,\R^m}_\cp, \F^{r,\M}_\cp),\quad \text{ with }\quad [\kbt]^{-1}=\kft. $$
Modifying the proof \cite[Proposition~3.6]{ShaoPre} in an obvious way and applying Lemma \ref{S3: involve time}, we have proved the statement of (b). 
\end{proof}
Recall that $\partial_\nu=\partial_{\nu_\M}$.
\begin{prop}
\label{S3: normal der}
$\Tu\partial_\nu \Tuei =\partial_\nu$ on $H^1_p(\Omega_i)$.
\end{prop}
\begin{proof}
It suffices to show that 
$\Tm\partial_\nu \Tmei=\partial_\nu$. On account of 
$$\kf(\partial_\nu u)=\kf(\nu_\M \cdot \nabla u)=\partial_{m+1} \kf \mathscr{E}_i u$$ 
for any $u\in H^1_p(\Omega_i)$, one readily computes
\begin{align*}
&\quad \Tm\partial_\nu \Tmei u = \Tm \partial_\nu \kbet \tmi \kfet \vge \mathscr{E}_i u +\Tm \partial_\nu (1-\vge)  u\\
&=\kbt\tm\kft \varsigma \partial_\nu \kbet \tmi \kfet \vge \mathscr{E}_i u + (1-\varsigma)\partial_\nu \kbet \tmi \kfet \vge \mathscr{E}_i u +\partial_\nu (1-\vge) u\\
&=\kbt\tm\kft \varsigma \kbt \partial_{m+1} \tmi \kfet \vge \mathscr{E}_i u + (1-\varsigma)\kbt\partial_{m+1}  \tmi \kfet \vge \mathscr{E}_i u +\partial_\nu (1-\vge) u\\
&=\kbt\tm\kft \varsigma \kbt \tmi \partial_{m+1} \kfet \vge \mathscr{E}_i u +\kbt\tm\kft(1-\varsigma)\kbt\tmi \partial_{m+1} \kfet \vge \mathscr{E}_i u\\
&\quad 
  +\partial_\nu (1-\vge)  u\\
&=\kbt \partial_{m+1} \kfet \vge \mathscr{E}_i  u +\partial_\nu (1-\vge)  u=\partial_\nu u.
\end{align*}
In the above, we have used the fact that $\Tm \partial_\nu (1-\vge) u=(1-\vge)   u$.
\end{proof}

Recall definitions~\eqref{S3: Vlme} and \eqref{S3: Vlm}.
The following proposition shows that the space $\ej$ is invariant under the transformation $\Tlme$ and $\Tlm$, respectively.
\begin{prop}
\label{S3: Omega-invariant spaces}
$[(\vartheta,h)\mapsto (\Tlme \vartheta,\Tlm h)]\in\Lis(\ej)\cap \Lis(\zej)$ with $(\lambda,\mu,\eta)\in\Br$. Moreover, there exist some $B_{\lambda,\mu,\eta}$, $B_{\lambda,\mu}$ satisfying 
\begin{align*}
\begin{cases}
[(\lambda,\mu,\eta)\mapsto B_{\lambda,\mu,\eta}]\in C^\omega(\Br, \L(\ea,\fa)),\\
[(\lambda,\mu)\mapsto B_{\lambda,\mu}]\in C^\omega(\Br,\L(\eb,\fb)),
\end{cases}
\end{align*}
such that
\begin{align*}
\begin{cases}
\partial_t [\vartheta_{\lambda,\mu,\eta}]=(1+\xi^\prime\lambda)\Tlme \partial_t \vartheta + B_{\lambda,\mu,\eta}(\vartheta_{\lambda,\mu,\eta}),\\
\partial_t[h_{\lambda,\mu}]=(1+\xi^{\prime}\lambda)\Tlm \partial_t h +B_{\lambda,\mu}(h_{\lambda,\mu}).
\end{cases}
\end{align*}
In particular, $B_{\lambda,0,0}=0$ and $B_{\lambda,0}=0$.  
\end{prop}
\begin{proof}
(i) Following the proof of \cite[Proposition~2.4]{EscPruSim03} and interpolation theory, we can show that for any Banach space $X$, $\rl\in{\Lis}(\F(I,X))$ and $[\rl]^{-1}=\varrho^{\lambda}_{\ast}$ with ${\F}\in\{BC^s,W^s_p\}$. In particular, there exists $M>0$ such that $\|\rl\|_{\L(\F(I,X))}\leq M$ for $\lambda\in{\B}$. A similar estimate as in \cite[Lemma~8.3]{EscPruSim03} by using the intrinsic norms of Besov spaces reveals that 
$$[(u,\rho)\mapsto(\tilde{u}_{\lambda,\mu,\eta},\tilde{\rho}_{\lambda,\mu}] \in \Lis(\mathbb{E}_1(J,\R^{m+1};E)\times\mathbb{E}_2(J,\R^m;E)).$$
(ii) Observe that
\begin{align}
\label{S3: time trace}
(\Tlme \vartheta,\Tlm h)(0)=(\vartheta,h)(0),\hspace{1em} \partial_t(\Tlme \vartheta,\Tlm h)(0)=\partial_t(\vartheta,h)(0),
\end{align}
whenever the corresponding derivative exists.
For any $z\in\Omega\setminus \Oet$, $\Tue \vartheta(z)=\vartheta(z)$, we conclude that $\partial_{\nu_{\Omega}}\Tue \vartheta=0$. For every $t\in J$, $\vartheta(t)\in BC(\Omega)$. It follows from Proposition~\ref{S3: Omega-continuity} that $\Tue(t) \vartheta(t)\in BC(\Omega)$, which implies that $[\![\Tue(t)\vartheta(t) ]\!]=0$. Moreover, for $ (\mu,\eta)\in\Br$
\begin{align*}
&\quad \|\Tue \vartheta\|_{\ea}=\sum_{i}\|\mathscr{R}_i\kbet\tTue \kfet(\vge \mathscr{E}_i \vartheta) + (1-\vge) \vartheta \|_{\ea}\\
&\leq M\sum_i  [\|\tTue\kfet(\vge \mathscr{E}_i \vartheta) \|_{\mathbb{E}_1(J,\R^{m+1};E)}+\|(1-\vge)\mathscr{E}_i\vartheta\|_{\ea}]\leq M \|\vartheta\|_{\ea}.
\end{align*}
The last line follows from (i), Lemma~\ref{S3: main lem} and point-wise multiplier theorem, see \cite[Section~3.3.2]{Trie78}. 
In virtue of \cite[Lemma~3.1]{ShaoPre}, one can show similarly that
$$[h\mapsto \Tu h]\in \L(\eb),\quad \mu\in\Br.$$
Then it is a direct consequence of the open mapping theorem that 
$$[(\vartheta,h)\mapsto (\Tlme \vartheta,\Tlm h)]\in\Lis(\ej).$$ 
By \eqref{S3: time trace}, the statement $[(\vartheta,h)\mapsto (\Tue \vartheta,\Tu h)]\in\Lis(\zej)$ is immediate.
\smallskip\\
(iii) The temporal derivative of $\vartheta_{\lambda,\mu,\eta}$ can be computed as follows.
\begin{align*}
&\quad \partial_t [\vartheta_{\lambda,\mu,\eta}]= (1-\vge)\partial_t[\rl \vartheta]
+ \sum_i \mathscr{R}_i\kbet\partial_t[\tTue \kfet(\vge \rl \mathscr{E}_i\vartheta)]\\
&=(1-\vge)(1+\xi^{\prime} \lambda)\rl \partial_t\vartheta
+\sum_i \{ \mathscr{R}_i\kbet\tTue \kfet(\vge \partial_t[\rl \mathscr{E}_i \vartheta])\\
&\quad +\mathscr{R}_i \sum\limits_{j=1}^m \kbet(\xi^\prime \chi_m\varpi \mu_j \tTue\partial_j[\kfet\vge\rl \mathscr{E}_i \vartheta] )\\
&\quad +\mathscr{R}_i\kbet(\xi^\prime \chi_m\chi\eta \tTue\partial_{m+1}[\kfet\vge\rl \mathscr{E}_i \vartheta]) \}\\
&=(1+\xi^\prime\lambda)\Tlme \partial_t\vartheta + B_{\lambda,\mu,\eta}(\vartheta_{\lambda,\mu,\eta}).
\end{align*}
The last line follows from the definition of $\Tlme$ and the equality
\begin{align*}
&\quad \mathscr{R}_i \kbet\tTue \kfet(\vge \partial_t[\rl \mathscr{E}_i \vartheta])+(1-\vge)(1+\xi^{\prime} \lambda)\rl \partial_t\vartheta\\
&=\mathscr{R}_i\kbet\tTue \kfet[\vge(1+\xi^\prime\lambda)\rl \partial_t \mathscr{E}_i \vartheta]+(1-\vge)(1+\xi^{\prime} \lambda)\rl \partial_t\vartheta\\
&=(1+\xi^\prime\lambda)\Tlme \partial_t\vartheta.
\end{align*}
Given $u\in H^1_p(\Omega\setminus\M)$, the term $B_{\lambda,\mu,\eta}$ can be explicitly written as
\begin{align*}
B_{\lambda,\mu,\eta}(u)&=\sum_i \{ \sum\limits_{j=1}^m \mathscr{R}_i\kbet(\xi^\prime \chi_m\varpi \mu_j \tTue\partial_j[\kfet\vge\Tuei \mathscr{E}_i u] )\}\\
&\quad+\mathscr{R}_i\kbet(\xi^\prime \chi_m\chi\eta \tTue\partial_{m+1}[\kfet\vge\Tuei \mathscr{E}_i u])\\
&=\sum_i \{ \sum\limits_{j=1}^m \mathscr{R}_i\xi^\prime \kbet\chi_m \kbet\varpi \mu_j \Tue\mathcal{A}^j\Tuei \mathscr{E}_i u \\
&\quad +\mathscr{R}_i \xi^\prime \kbet\chi_m \kbet\chi\eta \Tue\mathcal{A}^{m+1}\Tuei \mathscr{E}_i u \}.
\end{align*}
Here $\mathcal{A}^j$ with $j=1,\cdots,m+1$ are first order linear differential operators defined on $\Omega\setminus \M$ compactly supported in $\T_{a,{\kappa_c}}$ satisfying $\kfet\mathcal{A}^j u=\vge\partial_j \kfet u$ for all $u\in H^1_p(\Omega\setminus\M)$.
By means of Proposition~\ref{S3: Omega-Tue}(b) and the real analytic compatibility of the atlas $\Ae$ with the Euclidean structure, it is not hard to check that 
\begin{align}
\label{S3: analyticity of Blme}
[(\lambda,\mu,\eta)\mapsto B_{\lambda,\mu,\eta}]\in C^\omega(\Br, C^1(J;\L(W^{s+1}_p(\Omega\setminus\M),W^s_p(\Omega\setminus\M)))).
\end{align}
In \cite[Proposition~3.10]{ShaoPre}, it is shown that 
\begin{align*}
\partial_t[h_{\lambda,\mu}]=(1+\xi^{\prime}\lambda)\Tlm \partial_t h +B_{\lambda,\mu}(h_{\lambda,\mu})
\end{align*}
with $B_{\lambda,0}=0$ and
\begin{align}
\label{S3: analyticity of Blm}
[(\lambda,\mu)\mapsto B_{\lambda,\mu}]\in C^\omega(\Br,C^1(J;\L( W^{s+1}_p(\M), W^s_p(\M)) )).
\end{align}
This completes the proof.
\end{proof}
With minor modification of the above proof, we immediately have the following similar result.
\begin{prop}
\label{S3: Omega: invariant spaces 2}
For $(\lambda,\mu,\eta)\in\Br$ with sufficiently small $r_0>0$, 
$$[(f,g,q)\mapsto (f_{\lambda,\mu,\eta}, g_{\lambda,\mu}, q_{\lambda,\mu})]\in\Lis(\prod_{i=1}^3 \mathbb{F}_i(J),\prod_{i=1}^3 \mathbb{F}_i(J)),$$ 
and
$$[(f,g,q)\mapsto (f_{\lambda,\mu,\eta}, g_{\lambda,\mu}, q_{\lambda,\mu})]\in\Lis(\prod_{i=1}^3 \prescript{}{0}{\mathbb{F}}_i(J),\prod_{i=1}^3 \prescript{}{0}{\mathbb{F}_i}(J)).$$
\end{prop}
The main theorem of this section is
\begin{theorem}
\label{S3: main thm}
Suppose that 
$$(\vartheta,h)\in BC(J\times \Omega,E)\times BC(J \times \M,E) .$$ 
Then 
$$(\vartheta,h) \in C^k(\mathring{J}\times \T_{a/3}\!\setminus\!\M,E) \times C^k(\mathring{J}\times \M, E)$$
iff for any $(t_c,z_c) \in \mathring{J}\times \T_{a/3}\!\setminus\!\M$, there exists some $r_0=r_0(t_c,z_c)>0$ and a corresponding family of parameter-dependent diffeomorphisms $\{(\Tlme,\Tlm):(\lambda,\mu,\eta)\in\Br\}$ such that
\begin{align*}
[(\lambda,\mu,\eta)\mapsto (\vartheta_{\lambda,\mu,\eta},h_{\lambda,\mu})]\in C^k(\Br, BC(J\times \Omega,E)\times BC(J \times \M,E)).
\end{align*}
\end{theorem}
\begin{proof}
The proof follows in a similar manner to that of \cite[Theorem~3.5]{ShaoPre}.
\end{proof}
\medskip

\section{\bf Regularity of solutions to the Stefan Problem}

Throughout the rest of this paper, we assume that $\sigma>0$ and $d_i, \gamma \in C^{k+2}(0,\infty) $,   $\psi_i\in C^{k+3}(0,\infty)$  for $i=1,2$ with $k\in\N\cup \{\infty,\omega\}$ such that
$$\kappa_i(u)=-u\psi_i^{\prime\prime}(u)>0,\quad d_i(u)>0,\quad u\in(0,\infty),$$
and 
$$\gamma(u)>0,\quad u\in(0,\infty),\quad \text{ or }\quad \gamma\equiv 0.  $$
Let $J:=[0,T]$.
We define
\begin{align*}
&\gb:=H^1_p(J; W^{-1/p}_p(\M))\cap L_p(J; W^{2-1/p}_p(\M))\\
&\gc:=H^1_p(J; W^{-1-1/p}_p(\M))\cap L_p(J; W^{1-1/p}_p(\M)),
\end{align*}
and 
\begin{equation*}
\begin{aligned}
&\ca:=C(J\times\bar{\Omega})\cap C(J;C^1(\bar{\Omega}_i))\\
&\cb:=C(J;C^3(\M))\cap C^1(J;C(\M)).
\end{aligned}
\end{equation*}
It is shown in \cite[Section~3]{PruSimZac12} that
\begin{equation}
\label{S5: embed of ea,eb}
\ea\hookrightarrow \ca,\hspace{1em} \eb\hookrightarrow \cb.
\end{equation}
Observe that the constants in these embeddings blow up as $T\to 0$, however, they are uniform in $T$ if one considers the space $\zej$!

\subsection{\bf Linearization at a real analytic temperature and the initial interface}
Suppose that $\hat{z}_0=(\hat{\vartheta}_0,\hat{h}_0)\in\fd$ satisfies $ \hat{\vartheta}_0>0$.
Given any $\varepsilon>0$, by the Weierstrass approximation theorem, we can find some
$$\vartheta_A\in C^\omega(\bar{\Omega}),\quad \|\vartheta_A-\hat{\vartheta}_0\|_{BC(\Omega)}\leq \varepsilon.$$ 
For sufficiently small $\varepsilon$, $\vartheta_A>0$, and when $\gamma\equiv 0$, $l(\vartheta_A)\neq 0$. Let 
\begin{align*}
&\kappa_A(x)=\kappa(\vartheta_A(x)),\hspace{1em}d_A(x)=d(\vartheta_A(x)),\hspace{1em} l_A(x)=l(\vartheta_A(x)),\hspace{1em} \sigma_0=\frac{\sigma}{m},\\
&l_1(t,\cdot)=[\![\psi^\prime (e^{\Delta_\M t}\vartheta_A) ]\!],\hspace{1em} \gamma_1(t,\cdot)=\gamma(e^{\Delta_\M t}\vartheta_A).
\end{align*}
For a function $\vartheta\in \ea$, we do not distinguish $\vartheta|_\M$ from $\vartheta$ if the choice is self-evident from the context.
Here $\Delta_\M$ denotes the Laplace-Beltrami operator on $\M$. Similarly, we define 
\begin{align*}
&\hat{\kappa}_0(x)=\kappa(\hat{\vartheta}_0(x)),\hspace{1em}\hat{d}_0(x)=d(\hat{\vartheta}_0(x)),\hspace{1em} \hat{l}_0(x)=l(\hat{\vartheta}_0(x)),\hspace{1em}\hat{\gamma}_0=\gamma(\hat{\vartheta}_0),\\
&\hat{l}_1(t,\cdot)=[\![\psi^\prime (e^{\Delta_\M t}\hat{\vartheta}_0) ]\!],\hspace{1em} \hat{\gamma}_1(t,\cdot)=\gamma(e^{\Delta_\M t}\hat{\vartheta}_0).
\end{align*}
For $z=(\vartheta,h)\in\ej$, we define
\begin{align*}
\notag&F(z)=(\kappa_A-\kappa(\vartheta))\partial_t \vartheta +(d(\vartheta)-d_A)\Delta \vartheta -d(\vartheta)M_2(h):\nabla^2 \vartheta \\
\notag&\quad\quad \quad +d^\prime(\vartheta)|(I-M_1(h))\nabla \vartheta|^2 -d(\vartheta)(M_3(h)| \nabla \vartheta) +\kappa(\vartheta)\mathcal{R}(h)\vartheta,\\
\notag&G(z)= -([\![\psi(\vartheta)]\!] + \sigma {\cH}(h) ) +l_1 \vartheta +\sigma_0\Delta_\M h +(\gamma(\vartheta)\beta(h)-\gamma_1)\partial_t h,\\
\notag & Q(z)=[\![ (d(\vartheta)-d_A)\partial_\nu \vartheta ]\!] +(l_A -l(\vartheta))\partial_t h -([\![d(\vartheta)\nabla \vartheta]\!]| M_4(h)\nabla_\M h)\\
&\quad\quad\quad +\gamma(\vartheta)\beta(h)(\partial_t h)^2.
\end{align*}
Here we have
\begin{align*}
M_2(h)&=M_1(h)+M_1^{\sf T}(h)-M_1(h)M_1^{\sf T}(h),\\
M_3(h)&=(I-M_1(h)):\nabla M_1(h), \\
M_4(h)&=(I-M_1(h))^{\sf T}M_0(h).
\end{align*}
Employing the above notations and splitting into the principal linear part and a nonlinear part,
we arrive at the following formulation of problem \eqref{S2: transf Stefan}
\begin{equation}
\label{S5: Stefan-linear/nonlinear}
\left\{\begin{aligned}
\kappa_A \partial_t \vartheta -d_A \Delta \vartheta &=F(\vartheta,h)
 &&\text{in}&&\Omega\setminus\M,\\
\partial_{\nu_\Omega} \vartheta &=0
 &&\text{on}&&\partial \Omega,\\
[\![\vartheta]\!]&=0 &&\text{on}&&\M,\\
l_1(t )\vartheta + \sigma_0\Delta_\Sigma h- \gamma_1(t)\partial_t h &=G(\vartheta,h)
 &&\text{on}&&\M,\\
\mbox{}l_A \partial_t h -[\![d_A(x)\partial_\nu \vartheta]\!]&=Q(\vartheta,h)
 &&\text{on}&&\M,\\
\vartheta(0)=\hat{\vartheta}_0,\ h(0)&=\hat{h}_0.&&
\end{aligned}\right.
\end{equation}
We assume that $\hat{z}_0$ further satisfies the compatibility conditions ($\mathcal{CC}$), that is,
when $\gamma\equiv 0$
$$\partial_{\nu_\Omega}\hat{\vartheta}_0=0,\quad l_1(0)\hat{\vartheta}_0+\sigma_0\Delta_\M \hat{h}_0=G(\hat{z}_0), \quad Q(\hat{z}_0)+[\![ d_A\partial_{\nu} \hat{\vartheta}_0]\!]\in W^{2-6/p}_p(\M),$$
and when $\gamma> 0$
$$ \partial_{\nu_\Omega}\hat{\vartheta}_0=0,\quad l_A l_1(0)\hat{\vartheta}_0 +l_A\sigma_0\Delta_\M \hat{h}_0 -\gamma_1(0)[\![ d_A\partial_{\nu} \hat{\vartheta}_0]\!]=\gamma_1(0)Q(\hat{z}_0)+l_A G(\hat{z}_0).$$
In the definition of $(G(\hat{z}_0),Q(\hat{z}_0))$, it is understood that $\partial_t h(0)$ is replaced by
\begin{align}
\label{S5: replace h_t}
\partial_t h(0)=
\left\{\begin{aligned}
&\frac{1}{\hat{l}_0} ([\![\hat{d}_0 \partial_\nu \hat{\vartheta}_0]\!] -([\![\hat{d}_0 \nabla \hat{\vartheta}_0]\!]|M_4(\hat{h}_0)\nabla_\M \hat{h}_0 )), &&\gamma\equiv 0,\\
&\frac{1}{\beta(\hat{h}_0)\hat{\gamma}_0}([\![\psi(\hat{\vartheta}_0)]\!] +\sigma{\cH}(\hat{h}_0)),  &&\gamma>0. 
\end{aligned}\right.
\end{align}
When $\gamma\equiv 0$, we also need to impose the {\em well-posedness condition} $l(\hat{\vartheta}_0)\neq 0$ on $\Gamma_0$.

\subsection{\bf Regularity of a special solution}
In this section, we prove the analyticity of the solution to a linearized Stefan problem
with initial data $\hat{z}_0$, 
that is, we consider regularity of solutions to the following equation.
\begin{align}
\label{S5: reducing initial data}
\left\{\begin{aligned}
\kappa_A \partial_t \vartheta -d_A \Delta \vartheta &=0 &&\text{in}&& \Omega\setminus\M,\\
\partial_{\nu_\Omega}\vartheta &=0        &&\text{on}&& \partial\Omega,\\
[\![\vartheta]\!]&=0       &&\text{on}&& \M,\\
l_1 \vartheta+\sigma_0 \Delta_\M h - \gamma_1 \partial_t h &=g_1 &&\text{on}&& \M,\\
l_A \partial_t h -[\![ d_A\partial_\nu \vartheta]\!] &=q_1 &&\text{on}&& \M,\\
\vartheta(0)=\hat{\vartheta}_0,\hspace{1em} h(0)&=\hat{h}_0.
\end{aligned}\right.
\end{align}
Here $g_1:=e^{\Delta_\M t} G(\hat{z}_0)$ and $q_1:=e^{\Delta_\M t} Q(\hat{z}_0)$. 
Recall definition \eqref{S2: def of Fj}.
We set
\begin{align*}
&\fj=\{ (f,g,q,(\vartheta_0,h_0))\in \fa\times\fb\times\fc\times\fd:\\
&\quad\quad\quad (f,g,q,(\vartheta_0,h_0))\text{ satisfies the linear compatibility conditions }(\mathcal{LCC})\},
\end{align*}
where $(f,g,q,(\vartheta_0,h_0))$ is said to satisfy the linear compatibility conditions $(\mathcal{LCC})$ if
\begin{align*}
& \partial_{\nu_\Omega}\vartheta_0=0,\hspace{.5em} l_1(0)\vartheta_0+\sigma_0\Delta_\M h_0=g(0), \hspace{.5em} q(0)+[\![ d_A\partial_{\nu} \vartheta_0]\!]\in W^{2-6/p}_p(\M),\hspace{1.7em}\gamma\equiv 0,\\
&  \partial_{\nu_\Omega}\vartheta_0=0,\hspace{.5em} l_A l_1(0)\vartheta_0 +l_A\sigma_0\Delta_\M h_0 -\gamma_1(0)[\![ d_A\partial_{\nu} \vartheta_0]\!]=\gamma_1(0)q(0)+l_A g(0), 
\hspace{0.2em}\gamma> 0.
\end{align*}
We equipped $\fj$ with the following norm
\begin{align*}
\| (f,g,q,(\vartheta_0,h_0)) \|_{\fj}:=&\|f\|_{\fa}+\|g\|_{\fb} +\|q\|_{\fc} +\|(\vartheta_0, h_0)\|_{\fd}\\
& + (1-\text{sgn}(\gamma))\|q(0)+[\![ d_A\partial_{\nu} \vartheta_0]\!]\|_{W^{2-6/p}_p(\M)}.
\end{align*}

(i) Regularity of $g_1$ and $q_1$.
\smallskip\\
Case 1: $\gamma\equiv 0$. 
It is not hard to check that
\begin{align}
\label{S5: reg of k,d,gam,l}
l_1, \gamma_1\in \fb,\quad
\kappa_A\in W^1_p(\Omega\setminus\M),\hspace{.5em} d_A\in W^{2-2/p}_p(\Omega\setminus\M),\hspace{.5em} l_A\in W^{2-3/p}_p(\M).
\end{align}
Indeed, the first term can be obtained as follows. 
$\vartheta_A|_\M\in W^{2-3/p}_p(\M)$ and the $L_p$-maximal regularity of $\Delta_\M$ imply that 
\begin{align*}
e^{\Delta_\M t} \vartheta_A \in H^1_p(J; W^{-1/p}_p(\M))\cap L_p(J; W^{2-1/p}_p(\M)).
\end{align*}
By \cite[Proposition~3.2]{MeySch12}, we infer that 
$e^{\Delta_\M t} \vartheta_A \in \fb.$ Now the first term in \eqref{S5: reg of k,d,gam,l} follows from the regularity of $l,\gamma$.
Similar to \eqref{S5: reg of k,d,gam,l}, one checks that
\begin{align}
\label{S5: reg of k,d,gam,l-0}
\hat{l}_1, \hat{\gamma}_1\in \fb, \hspace{.5em}\hat{\kappa}_0\in W^1_p(\Omega\setminus\M),\hspace{.5em}\hat{d}_0\in W^{2-2/p}_p(\Omega\setminus\M),\hspace{.5em} \hat{l}_0\in W^{2-3/p}_p(\M).
\end{align}
In \cite{Shao13}, an analysis of the structure of the mean curvature operator ${\cH}(h)$ and $\beta(h)$ is obtained, which shows that ${\cH}(h)$ is a rational function in the height function $h$ and its spatial derivatives up to second order, while $\beta(h)$ is a rational function in $h$ and its first order spatial derivatives. On account of the embeddings $W^{3-3/p}_p(\M)\hookrightarrow C(\M)$ and $ W^{2-3/p}_p(\M)\hookrightarrow C(\M)$, we conclude from \cite[Theorem~2.8.3]{Trie78} and a similar argument to \cite[Proposition~2.7]{ShaoSim13} that $W^{3-3/p}_p(\M)$ and $W^{2-3/p}_p(\M)$ are multiplication algebras under point-wise multiplication. Since
$$[x\mapsto x^a]\in C^{\omega}((0,\infty)),\quad a\in\R,$$
well-known results for substitution operators for Sobolev-Slobodeckii spaces imply that
\begin{align}
\label{S5: reg of H & beta}
\beta(\hat{h}_0) \in W^{3-3/p}_p(\M),\quad {\cH}(\hat{h}_0)\in W^{2-3/p}_p(\M).
\end{align}
\eqref{S5: reg of k,d,gam,l}-\eqref{S5: reg of H & beta} then imply  
\begin{align}
\label{S5: reg of G(z0)}
G(\hat{z}_0)\in W^{2-3/p}_p(\M),
\end{align}
and a similar argument yields 
\begin{align}
\label{S5: reg of Q(z0)}
Q(\hat{z}_0)\in W^{1-3/p}_p(\M).
\end{align}
Together with the $L_p$-maximal regularity of $\Delta_\M$ and \cite[Proposition~3.2]{MeySch12}, we conclude from \eqref{S5: reg of G(z0)} and \eqref{S5: reg of Q(z0)} that
\begin{align}
\label{S5: reg of g_1,q_1}
g_1\in\gb\hookrightarrow \fb,\hspace{1em} q_1\in\gc\hookrightarrow \fc.
\end{align}
Case 2: $\gamma>0$.
Based on the discussion in (i), it suffices to show the regularity of $\partial_t h(0)$, which is defined in \eqref{S5: replace h_t}. As illustrated in (i),  we have 
$$\beta(\hat{h}_0)\hat{\gamma}_0, ([\![\psi(\hat{\vartheta}_0)]\!] +\sigma{\cH}(\hat{h}_0))\in W^{2-3/p}_p(\M).$$ 
It implies that $\partial_t h(0)\in W^{2-3/p}_p(\M)$. It yields the desired results, i.e.,  
$$g_1\in \fb,\hspace{1em} q_1\in \fc.$$
In virtue of condition ($\mathcal{CC}$) and the definitions of $g_1$ and $q_1$, one checks that condition ($\mathcal{LCC}$) is at our disposal.
Therefore, all the compatibility conditions in \cite[Theorems~3.3, 3.5]{PruSimZac12} are satisfied, and then there exists a unique solution $z^* =(\vartheta^* ,h^* )\in\ej$ to the linear system \eqref{S5: reducing initial data}. 
\begin{remark}
The compatibility condition $[\![d(\theta_0)\partial_{\nu_{\Gamma_0}} \theta_0]\!]\in W^{2-6/p}_p(\Gamma_0)$ in Theorem~\ref{S1: main theorem: gamma=0} implies that 
$$[\![ \hat{d}_0 \partial_\nu \hat{\vartheta}_0]\!] - ([\![ \hat{d}_0 \nabla \hat{\vartheta}_0]\!]| M_4(\hat{h}_0) \nabla_\Sigma \hat{h}_0) \in W^{2-6/p}_p(\Sigma) , $$
which is equivalent to
$$([\![ \hat{d}_0 \nabla \hat{\vartheta}_0]\!]| \nu_\Sigma - M_4(\hat{h}_0) \nabla_\Sigma \hat{h}_0) \in W^{2-6/p}_p(\Sigma) .$$
From the above discussion, 
we infer that
$$[\![ \nabla \hat{\vartheta}_0]\!] \in  W^{2-6/p}_p(\Sigma, \R^{m+1}) . $$
Now based on this observation and \eqref{S5: reg of k,d,gam,l-0}, we conclude that
$$Q(\hat{z}_0)+[\![ d_A\partial_{\nu} \hat{\vartheta}_0]\!]\in W^{2-6/p}_p(\M). $$
The other two conditions in conditions ($\mathcal{CC}$), i.e.,
when $\gamma\equiv 0$
$$\partial_{\nu_\Omega}\hat{\vartheta}_0=0,\quad l_1(0)\hat{\vartheta}_0+\sigma_0\Delta_\M \hat{h}_0=G(\hat{z}_0),$$
and when $\gamma> 0$
$$ \partial_{\nu_\Omega}\hat{\vartheta}_0=0,\quad l_A l_1(0)\hat{\vartheta}_0 +l_A\sigma_0\Delta_\M \hat{h}_0 -\gamma_1(0)[\![ d_A\partial_{\nu} \hat{\vartheta}_0]\!]=\gamma_1(0)Q(\hat{z}_0)+l_A G(\hat{z}_0).$$
can be easily obtained from the remaining compatibility conditions in Theorems~\ref{S1: main theorem: gamma=0} and \ref{S1: main theorem: gamma>0}. Similarly, the compatibility conditions in Theorems~\ref{S1: main theorem: gamma=0} and \ref{S1: main theorem: gamma>0} can also be concluded from conditions ($\mathcal{CC}$). By  \cite[Theorems~3.1, 3.2]{PruSimZac12}, problem \eqref{S5: Stefan-linear/nonlinear} with initial data $\hat{z}_0$ has a unique $L_p$-solution on some possibly small but nontrivial interval $J:=[0,T]$, which is denoted by $\hat{z}=(\hat{\vartheta},\hat{h})$. 
\qed
\end{remark}

Next, we apply the parameter-dependent diffeomorphisms $\Tlme$ and $\Tlm$ to show the analyticity of the solution $z^* $. We will use the following useful fact that for any time-independent map $\mathcal{N}$ acting on $\F^s(\Omega_i,E)$
$$\Tlme\mathcal{N}=\Tue\mathcal{N}\Tuei \Tlme  ,$$
and a similar result also holds for $\Tlm$.
By Proposition~\ref{S3: Omega-invariant spaces}, we have
\begin{align*}
\partial_t [\vartheta^*_{\lambda,\mu,\eta}]&=(1+\xi^\prime\lambda)\Tlme \partial_t\vartheta^* + B_{\lambda,\mu,\eta}(\vartheta^*_{\lambda,\mu,\eta})\\
&=(1+\xi^\prime\lambda)\Tlme (d_A/\kappa_A)\Delta \vartheta^*  +B_{\lambda,\mu,\eta}(\vartheta^*_{\lambda,\mu,\eta})\\
&=(1+\xi^\prime\lambda) (d_{A,\lambda,\mu,\eta}/\kappa_{A,\lambda,\mu,\eta})\Tue\Delta\Tuei \vartheta^*_{\lambda,\mu,\eta}+B_{\lambda,\mu,\eta}(\vartheta^*_{\lambda,\mu,\eta}),
\end{align*}
and either by Proposition~\ref{S3: normal der}
\begin{align*}
\partial_t [h^*_{\lambda,\mu}]&=(1+\xi^{\prime}\lambda)\Tlm \partial_t h^* +B_{\lambda,\mu}(h^*_{\lambda,\mu})\\
&=(1+\xi^{\prime}\lambda)\Tlm [(q_1+[\![d_A\partial_\nu \vartheta^* ]\!])/l_A] +B_{\lambda,\mu}(h^*_{\lambda,\mu})\\
&=(1+\xi^{\prime}\lambda)\{(q_{1,\lambda,\mu}/l_{A,\lambda,\mu})+[\![(d_{A,\lambda,\mu}/l_{A,\lambda,\mu}) \partial_\nu \vartheta^*_{\lambda,\mu}]\!] \}
+B_{\lambda,\mu}(h^*_{\lambda,\mu}),
\end{align*}
when $\gamma\equiv 0$, or when $\gamma>0$, we have
\begin{align*}
\partial_t [h^*_{\lambda,\mu}]&=(1+\xi^{\prime}\lambda)\Tlm [(l_1 \vartheta^* +\sigma_0\Delta_\M h^* -g_1 )/\gamma_1] +B_{\lambda,\mu}(h^*_{\lambda,\mu})\\
&=(1+\xi^{\prime}\lambda)(l_{1,\lambda,\mu} \vartheta^*_{\lambda,\mu} +\sigma_0\Tu\Delta_\M\Tui h^*_{\lambda,\mu}- g_{1,\lambda,\mu} )/( \gamma_{1,\lambda,\mu})
+B_{\lambda,\mu}(h^*_{\lambda,\mu}).
\end{align*}
We define a map $\Phi:\ej\times\Br\rightarrow \fj$: $((\vartheta,h),(\lambda,\mu,\eta))\mapsto$
\begin{align*}
\begin{cases}
\kappa_{A,\lambda,\mu,\eta}\partial_t \vartheta -(1+\xi^\prime\lambda) d_{A,\lambda,\mu,\eta}\Tue\Delta\Tuei \vartheta-\kappa_{A,\lambda,\mu,\eta} B_{\lambda,\mu,\eta}(\vartheta)
\hspace{2em} &\text{in }\quad \Omega\setminus\M,\\
(1+{\rm sgn}(\gamma) \xi^{\prime}\lambda)(l_{1,\lambda,\mu} \vartheta +\sigma_0\Tu\Delta_\M\Tui h-g_{1,\lambda,\mu})-\gamma_{1,\lambda,\mu}\partial_t h\\
\quad\quad +\gamma_{1,\lambda,\mu} B_{\lambda,\mu}(h) &\text{on }\quad\M,\\
l_{A,\lambda,\mu}\partial_t h-(1+\xi^{\prime}\lambda)(q_{1,\lambda,\mu}+[\![d_{A,\lambda,\mu} \partial_\nu \vartheta]\!]) -l_{A,\lambda,\mu}B_{\lambda,\mu}(h) &\text{on }\quad\M,\\
\vartheta(0)-\hat{\vartheta}_0 &\text{in }\quad\Omega\setminus\M,\\
h(0)-\hat{h}_0 &\text{on }\quad\M.
\end{cases}
\end{align*}
Setting $z^*_{\lambda,\mu,\eta}=(\vartheta^*_{\lambda,\mu,\eta},h^*_{\lambda,\mu})$, note that $\Phi(z^*_{\lambda,\mu,\eta},(\lambda,\mu,\eta))=0$ for all $(\lambda,\mu,\eta)\in\Br$. 

(ii) We need to show that $\Phi$ actually maps into $\fj$. For simplification, we set $\Phi(z,(\lambda,\mu,\eta))=(f_2,g_2,q_2,(\vartheta_0,h_0))^{\sf T}$. It is obvious that based on trace theory of anisotropic Sobolev-Slobodeckii spaces,
see \cite[Section~2]{DenPruZac08},
\begin{align*}
(\vartheta_0, h_0) \in \fd.
\end{align*}
By means of Propositions~\ref{S3: Omega-Tue} and \ref{S3: Omega-invariant spaces}, we infer that $f_2 \in \fa$.
For the regularity of the next two terms, we split into two cases as before.
When $\gamma\equiv 0$, observe that $\Delta_\M \in \L(\eb,\fb)$. Proposition~\ref{S3: Omega: invariant spaces 2}, \eqref{S5: reg of k,d,gam,l} and \eqref{S5: reg of g_1,q_1} yield 
\begin{align}
\label{S5: reg of trans of g_1,q_1,l_1,gam_1}
\gamma_{1,\lambda,\mu}\, ,\, l_{1,\lambda,\mu}\, ,\, g_{1,\lambda,\mu}\in \fb, \hspace{1em}  q_{1,\lambda,\mu}\in\fc.
\end{align}
Taking into consideration that $\partial_t h\in \fc$ and the fact that $\fb,\fc$ are multiplication algebras, we conclude from Propositions~\ref{S3：M-reg},~\ref{S3: Omega-invariant spaces} and \eqref{S5: reg of trans of g_1,q_1,l_1,gam_1} that
\begin{align*}
g_2 \in \fb,\hspace{1em} q_2\in \fc.
\end{align*}
When $\gamma>0$, since $\partial_t h\in\fb$, the desired regularity results clearly hold true.
It remains to show condition ($\mathcal{LCC}$). 
\smallskip\\
Case 1: $\gamma\equiv 0$. It is immediate that $\partial_{\nu_\Omega} \vartheta_0 = \partial_{\nu_\Omega}(\vartheta(0)-\hat{\vartheta}_0)=0$. One checks that
\begin{align*}
&\quad l_1(0)\vartheta_0+\sigma_0\Delta_\M h_0 = l_1(0)(\vartheta(0)-\hat{\vartheta}_0)+\sigma_0\Delta_\M (h(0)-\hat{h}_0)\\
&=l_{1,\lambda,\mu}(0)\vartheta(0)+\sigma\Tu(0)\Delta_\M \Tui(0)h(0)-g_{1,\lambda,\mu}(0)=g_2(0),
\end{align*} 
by recalling $u_{\lambda,\mu}(0,\cdot)=u(0,\cdot)$ for any $u\in E^\Sigma$, and
\begin{align*}
q_2(0)+[\![ d_{A,\lambda,\mu }(0) \partial_\nu \vartheta_0]\!]&= l_{A,\lambda,\mu }(0)\partial_t h(0) -q_{1,\lambda,\mu }(0)-[\![ d_A \partial_\nu  \vartheta(0)]\!] \\
&\quad -l_{A,\lambda,\mu }(0)B_{\lambda,\mu}(h(0)) +[\![ d_A\partial_\nu (\vartheta(0)- \hat{\vartheta}_0)]\!]\\
&=l_{A}\partial_t h(0)-q_{1}(0) -l_{A}B_{\lambda,\mu}(h(0))-[\![ d_A\partial_\nu \hat{\vartheta}_0]\!].
\end{align*}
By the discussion in (ii), $q_{1}(0) +[\![ d_A\partial_\nu \hat{\vartheta}_0]\!]\in W^{2-6/p}_p(\M)$. Since $W^{2-6/p}_p(\M)$ is a multiplication algebra, \eqref{S3: analyticity of Blm}, \eqref{S5: reg of k,d,gam,l} and \cite[formula~(3.1)]{PruSimZac12} imply that 
\begin{align*}
q_2(0)+[\![ d_{A,\lambda,\mu }(0) \partial_\nu \vartheta_0]\!]\in W^{2-6/p}_p(\M).
\end{align*}
Case 2: $\gamma>0$.
\begin{align*}
&\quad l_A l_1(0)\vartheta_0 + l_A\sigma_0 \Delta_\M h_0 -\gamma_1(0)[\![ d_A \partial_\nu \vartheta_0 ]\!]\\
&=l_{A,\lambda,\mu}(0) l_{1,\lambda,\mu}(0)(\vartheta(0)-\hat{\vartheta}_0)+ l_{A,\lambda,\mu}(0)\sigma_0 \Tu(0)\Delta_\M \Tui(0)(h(0)-\hat{h}_0)\\
&\quad -\gamma_{1,\lambda,\mu}(0)[\![ d_{A,\lambda,\mu}  \partial_\nu  (\vartheta(0)-\hat{\vartheta}_0) ]\!]\\
&=l_A l_1(0) \vartheta(0) + l_A\sigma_0 \Delta_\M  h(0) -\gamma_1(0)[\![ d_A \partial_\nu \vartheta(0) ]\!] 
 -l_A g_1(0) -\gamma_1(0)q_1(0)\\
&=l_A g_2(0) +\gamma_1(0)q_2(0).
\end{align*}
Therefore $\Phi(z,(\lambda,\mu,\eta))\in\fj$.

(iii) Let $w=(u,\rho)\in\ej$. The Fr\`echet derivative of $\Phi$ with respect to $(\vartheta,h)$ at $(z^* ,0)$ is clearly given by 
\begin{align*}
D_1 \Phi(z^* ,0)w=
\left\{\begin{aligned}
&\kappa_A\partial_t u -d_A\Delta u &&\text{in}&& \Omega\setminus\M,\\
&l_1 u +\sigma_0\Delta_\M\rho -\gamma_1\partial_t \rho &&\text{on}&& \M,\\
&l_A \partial_t \rho -[\![d_A\partial_\nu u]\!]   &&\text{on}&& \M,\\
&u(0) &&\text{in}&& \Omega\setminus\M,\\
&\rho(0) &&\text{on}&& \M.
\end{aligned}\right.
\end{align*}
It is obvious that
\begin{align*}
D_1 \Phi(z^* ,0)w\in \fa\times\fb\times\fc\times\fd.
\end{align*}
Condition ($\mathcal{LCC}$) can be verified as in (ii) by using \cite[formula~(3.1)]{PruSimZac12} in the case $\gamma\equiv 0$. We can deduce from \cite[Theorems~3.3, 3.5]{PruSimZac12} that
\begin{align*}
D_1 \Phi(z^* ,0) \in \Lis(\ej,\fj).
\end{align*}
\smallskip\\
(iv) Regularity of $[(\lambda,\mu)\mapsto (g_{1,\lambda,\mu},q_{1,\lambda,\mu})]$. We express $g_{1,\lambda,\mu}$ as
\begin{align*}
g_{1,\lambda,\mu}&=\Tlm e^{\Delta_\M t}G(\hat{z}_0)=\Tlm(c+\Delta_\M)e^{\Delta_\M t}(c+\Delta_\M)^{-1}G(\hat{z}_0)\\
&=\Tu (c+\Delta_\M)\Tui \Tlm e^{\Delta_\M t}(c+\Delta_\M)^{-1}G(\hat{z}_0).
\end{align*}
For sufficiently large $c$, it is well-known that $(c+\Delta_\M)$ is an isomorphism from $W^{s+2}_p(\M)$ to $W^s_p(\M)$ for any $s\in \R$. We consider the solution to
\begin{align}
\label{S5: eq of g_1}
\partial_t \rho+ \Delta_\M \rho=0, \quad \rho(0)=(c+\Delta_\M)^{-1} G(\hat{z}_0).
\end{align}
$\hat{\rho}:=e^{\Delta_\M t}(c+\Delta_\M)^{-1}G(\hat{z}_0)$ is the unique solution to \eqref{S5: eq of g_1}. Furthermore, Proposition~\ref{S3: Omega-invariant spaces} shows that $\hat{\rho}_{\lambda,\mu}$ satisfies
\begin{align*}
\partial_t [\hat{\rho}_{\lambda,\mu}]&= (1+\xi^\prime \lambda)\Tlm \partial_t \hat{\rho} + B_{\lambda,\mu}(\hat{\rho}_{\lambda,\mu})\\
&= -(1+\xi^\prime \lambda)\Tu\Delta_\M \Tui \hat{\rho}_{\lambda,\mu} + B_{\lambda,\mu}(\hat{\rho}_{\lambda,\mu}).
\end{align*}
We define the map $\Phi_g: \mathbb{X}_2(J) \times \Br \rightarrow L_p(J;W^{2-1/p}_p(\M))\times W^{4-3/p}_p(\M)$ by
\begin{align*}
\Phi_g(\rho,(\lambda,\mu))=\binom{\partial_t \rho +(1+\xi^\prime \lambda)\Tu\Delta_\M \Tui \rho -B_{\lambda,\mu}(\rho)}{\rho(0)-(c+\Delta_\M)^{-1} G(\hat{z}_0)}.
\end{align*}
Here $\mathbb{X}_2(J):=H^1_p(J; W^{2-1/p}(\M))\cap L_p(J; W^{4-1/p}(\M))$.
Note that $\Phi_g(\hat{\rho}_{\lambda,\mu}, (\lambda,\mu))=(0,0)^{\sf T}.$
By the $L_p$-maximal regularity of $\Delta_\M$, we immediately have
\begin{align*}
D_1\Phi_g(\hat{\rho},0)\in \Lis(\mathbb{X}_2(J), L_p(J;W^{2-1/p}_p(\M))\times W^{4-3/p}_p(\M)).
\end{align*}
We define a bilinear and continuous map 
\begin{align*}
T:  C(J;\L(W^{4-1/p}_p(\M), W^{3-1/p}_p(\M)))\times \mathbb{X}_2(J)\rightarrow L_p(J;W^{2-1/p}_p(\M))
\end{align*}
by $(B,u)\mapsto [t\mapsto B(t)u(t)]$. Since $T$ is real analytic, by \eqref{S3: analyticity of Blm}, we get
\begin{align*}
[(u,(\lambda,\mu))\mapsto B_{\lambda,\mu}(u)]\in C^\omega (\mathbb{X}_2 \times \Br, L_p(J;W^{2-1/p}_p(\M))). 
\end{align*}
In virtue of Proposition~\ref{S3：M-reg}(b) and the above bilinear map argument, one gets
\begin{align*}
\Phi_g \in C^\omega (\mathbb{X}_2 \times \Br, L_p(J;W^{2-1/p}_p(\M))\times W^{4-3/p}_p(\M)).
\end{align*}
By the implicit function theorem, there exists some $\mathbb{B}(0,r_1)\subset\Br$ such that 
\begin{align*}
[(\lambda,\mu)\mapsto \Tlm e^{\Delta_\M t}(c+\Delta_\M)^{-1}G(\hat{z}_0)] \in C^\omega(\mathbb{B}(0,r_1), \mathbb{X}_2).
\end{align*}
Without loss of generality, we may assume that $r_1=r_0$. In view of Proposition~\ref{S3：M-reg}(b), we have that for all $s\in\R$  
\begin{align*}
[\mu\mapsto \Tu (c+\Delta_\M)\Tui]\in C^{\omega}(\Br, C^1(J; \L(W^{s+2}_p(\M), W^s_p(\M)))).
\end{align*}
The above bilinear map argument and the embedding $\gb\hookrightarrow \fb$ yield 
\begin{align}
\label{S5: ana of g_1-trans}
[(\lambda,\mu)\mapsto g_{1,\lambda,\mu}] \in C^{\omega} (\Br, \fb).
\end{align}
Following a similar discussion, one obtains
\begin{align}
\label{S5: ana of q_1-trans}
[(\lambda,\mu)\mapsto q_{1,\lambda,\mu}] \in C^{\omega} (\Br, \fc).
\end{align}
\smallskip\\
(v) Regularity of the map $\Phi$. It follows from Proposition~\ref{S3: Omega-Tue} that
\begin{align}
\label{S5: ana of ka,d-trans}
[(\lambda,\mu,\eta)\mapsto (\kappa_{A,\lambda,\mu,\eta}, d_{A,\lambda,\mu,\eta})] \in C^k (\Br, C^1(J; BC(\Omega))^2).
\end{align}
Proposition~\ref{S3：M-reg} implies that  
\begin{align}
\label{S5: ana of l,d-trans}
[(\lambda,\mu)\mapsto (l_{A,\lambda,\mu}, d_{A,\lambda,\mu})] \in C^k (\Br, C^1(J; BC^1(\M))^2).
\end{align}
Similarly, we have
\begin{align}
\label{S5: ana of l_1,ga_1-trans}
[(\lambda,\mu)\mapsto (l_{1,\lambda,\mu}, \gamma_{1,\lambda,\mu})] \in C^k (\Br, (C^1(J; BC(\M)) \cap C(J; BC^2(\M)))^2).
\end{align}
By Proposition~\ref{S3: Omega-Tue}, one concludes that 
$$[(\mu,\eta)\mapsto \Tue \Delta \Tuei]\in C^\omega (\Br, C^1(J; \L(H^2_p(\Omega\setminus \M), L_p(\Omega))) ).$$
Combined with \eqref{S5: ana of ka,d-trans}, this yields
$$[(\lambda,\mu,\eta)\mapsto d_{A,\lambda,\mu,\eta}\Tue \Delta \Tuei] \in C^k (\Br, \L(\ea,\fa) ).$$
It follows again from \eqref{S5: ana of ka,d-trans} that
$$[(\vartheta, (\lambda,\mu,\eta))\mapsto \kappa_{A,\lambda,\mu,\eta}\partial_t \vartheta]\in C^k (\ea\times\Br, \fa).$$
Together with Proposition~\ref{S3: Omega-invariant spaces}, the above discussion shows that
\begin{align*}
[(\vartheta,(\lambda,\mu,\eta))& \mapsto \kappa_{A,\lambda,\mu,\eta}\partial_t \vartheta -(1+\xi^\prime\lambda) d_{A,\lambda,\mu,\eta}\Tue\Delta\Tuei \vartheta-B_{\lambda,\mu,\eta}(\vartheta)]\\
& \in C^k(\ea\times\Br,\fa).
\end{align*}
Applying \eqref{S5: ana of g_1-trans}-\eqref{S5: ana of l_1,ga_1-trans}, one can check the regularity of all the other entries of $\Phi$, and thus
$$\Phi\in C^k(\ej\times\Br,\fj).$$
Now the implicit function theorem yields
\begin{align*}
[(\lambda,\mu,\eta)\mapsto z^*_{\lambda,\mu,\eta}]\in C^k(\Br, \ej).
\end{align*}
As a conclusion of Theorem~\ref{S3: main thm}, we obtain
\begin{align*}
z^* \in C^k(\mathring{J}\times \T_{a/3}\!\setminus\!\M) \times C^k (\mathring{J}\times\M).
\end{align*}
To attain the regularity of $\vartheta^* $ in $\Omega\setminus \bar{\T}_{a/6}$, we study the linear equations
\begin{equation*}
\left\{\begin{aligned}
\kappa_1(\vartheta_A) \partial_t \vartheta -d_1(\vartheta_A)\Delta\vartheta&=0 &&\text{in}&&\Omega_1\setminus\bar{\T}_{a/6}\\
\vartheta&=\vartheta^* &&\text{on}&&\partial_1\T_{a/6}\\
\vartheta(0)&=\hat{\vartheta}_0 &&\text{in}&&\Omega_1\setminus\bar{\T}_{a/6},
\end{aligned}\right.
\end{equation*}
and
\begin{equation*}
\left\{\begin{aligned}
\kappa_2(\vartheta_A) \partial_t \vartheta -d_2(\vartheta_A)\Delta\vartheta&=0 &&\text{in}&&\Omega_2\setminus\bar{\T}_{a/6}\\
\partial_{\nu_\Omega} \vartheta &=0 &&\text{on}&&\partial \Omega \\
\vartheta&=\vartheta^* &&\text{on}&&\partial_2\T_{a/6}\\
\vartheta(0)&=\hat{\vartheta}_0 &&\text{in}&&\Omega_2\setminus\bar{\T}_{a/6},
\end{aligned}\right.
\end{equation*}
where $\partial_i\T_{a/6}:=\Lambda(\M,(-1)^ia/6)$. Since $\vartheta^*$ is analytic on $\partial_i\T_{a/6}$, we can obtain the regularity of  $\vartheta^*$ in these two domains by means of the parameter-dependent diffeomorphism defined in \cite{EscPruSim03}, the results in \cite[Section~8]{DenHiePru03}, and the implicit function theorem as above. 
To sum up, we have established the analyticity of the solution $z^* $, i.e.,
\begin{align}
\label{S5: ana of z*}
z^* \in C^k(\mathring{J}\times \Omega\setminus\M) \times C^k (\mathring{J}\times\M).
\end{align}
 
\subsection{\bf Reduction to zero initial data}
Recall that $\hat{z}=(\hat{\vartheta},\hat{h})$ is the unique $L_p$ solution to problem \eqref{S5: Stefan-linear/nonlinear} with initial data $\hat{z}_0$.
We set $\bar{\vartheta}=\hat{\vartheta}-\vartheta^*$, $\bar{h}=\hat{h}-h^*$. 
Then $\bar{z}=(\bar{\vartheta},\bar{h})$ is the unique solution to the following equation.
\begin{equation}
\label{S5: equation-zero intial}
\left\{\begin{aligned}
\kappa_A \partial_t \vartheta -d_A\Delta \vartheta&=F(\vartheta+\vartheta^*,h+h^*)
&&\text{in}&&\Omega\setminus\M, \\
\partial_{\nu_\Omega} \vartheta&=0 &&\text{on}&&\partial \Omega,\\
[\![\vartheta]\!]&=0 &&\text{on}&& \M,\\
l_1 \vartheta + \sigma_0 \Delta_\M h- \gamma_1\partial_t h
&=\bar{G}(\vartheta,h;\vartheta^* ,h^* ) &&\text{on}&& \M,\\
l_A \partial_t h -[\![d_A\partial_\nu \vartheta]\!]
&=\bar{Q}(\vartheta,h; \vartheta^* ,h^* ) &&\text{on}&& \M,\\
\vartheta(0)&=0 &&\text{in}&&\Omega\setminus\M, \\
h(0)&=0 &&\text{on}&& \M.
\end{aligned}\right.
\end{equation}
Here we have set
\begin{equation*}
\begin{split}
&\bar{G}(\vartheta,h;\vartheta^* ,h^*)=G(\vartheta+\vartheta^* ,h+h^* )-e^{\Delta_\M t}G(\hat{z}_0),\\
&\bar{Q}(\vartheta,h;\vartheta^* ,h^*)
= Q(\vartheta+\vartheta^* ,h+h^* )-e^{\Delta_\M t}Q(\hat{z}_0).\\
\end{split}
\end{equation*}
Note that $\bar{G}(0,0;\hat{\vartheta}_0,\hat{h}_0)=\bar{Q}(0,0;\hat{\vartheta}_0,\hat{h}_0)=0$
by construction, which ensures time trace zero at $t=0$. 

Let $z^*_{\lambda,\mu,\eta}:=(\vartheta^*_{\lambda,\mu,\eta},h^*_{\lambda,\mu})$.
As in Section~5.1, we compute the temporal derivative of $\bar{z}_{\lambda,\mu,\eta}:=(\bar{\vartheta}_{\lambda,\mu,\eta},\bar{h}_{\lambda,\mu})$ as follows.
\begin{align*}
\partial_t[\bar{\vartheta}_{\lambda,\mu,\eta}]&=(1+\xi^\prime \lambda)(d_{A,\lambda,\mu,\eta}/\kappa_{A,\lambda,\mu,\eta})\Tue\Delta\Tuei \bar{\vartheta}_{\lambda,\mu,\eta } +B_{\lambda,\mu,\eta}(\bar{\vartheta}_{\lambda,\mu,\eta})\\
&\quad +F_{\lambda,\mu,\eta}(\bar{z}_{\lambda,\mu,\eta},z^*_{\lambda,\mu,\eta})/\kappa_{A,\lambda,\mu,\eta},
\end{align*}
and either 
\begin{align*}
\partial_t[\bar{h}_{\lambda,\mu}]=(1+\xi^\prime \lambda)[\![ (d_{A,\lambda,\mu}/l_{A,\lambda,\mu})\partial_\nu \bar{\vartheta}]\!] + B_{\lambda,\mu}(\bar{h}_{\lambda,\mu})+\bar{Q}_{\lambda,\mu}(\bar{z}_{\lambda,\mu,\eta},z^*_{\lambda,\mu,\eta})/l_{A,\lambda,\mu},
\end{align*}
when $\gamma\equiv 0$, or when $\gamma>0$ we have
\begin{align*}
\partial_t[\bar{h}_{\lambda,\mu}]&=(1+\xi^\prime \lambda)(l_{1,\lambda,\mu}\bar{\vartheta}_{\lambda,\mu} +\sigma_0\Tu\Delta_\M \Tui \bar{h}_{\lambda,\mu})/\gamma_{1,\lambda,\mu}
+B_{\lambda,\mu}(\bar{h}_{\lambda,\mu})\\
&\quad -\bar{G}_{\lambda,\mu}(\bar{z}_{\lambda,\mu,\eta},z^*_{\lambda,\mu,\eta})/\gamma_{1,\lambda,\mu}.
\end{align*}
For sufficiently small $r_0$, 
$\Theta_{\mu,\eta}(z)=z$ for $(\mu,\eta)\in \Br$ and $z\in\Omega\setminus\T_{a/3}$. Hence, $(\zeta\circ d_\M)_{\lambda,\mu,\eta}=\zeta\circ d_\M$.
In virtue of $(\zeta\circ d_\M)_{\lambda,\mu,\eta}(x)=0$ for $x\notin\Omega\setminus\T_{2a/3}$, and $(h\circ \Pi)_{\lambda,\mu,\eta}(x)=(h_{\lambda,\mu}\circ \Pi)(x)$ for $x\in \T_{2a/3}$, 
one readily verifies that
\begin{align*}
\Tlme \Upsilon(h)&=(\zeta\circ d_\M)_{\lambda,\mu,\eta}(h_{\lambda,\mu}\circ\Pi)_{\lambda,\mu,\eta}(\nu_\M\circ \Pi)_{\lambda,\mu,\eta}\\
&= (\zeta\circ d_\M)(h_{\lambda,\mu}\circ\Pi)(\nu_\M\circ \Pi)_{\lambda,\mu,\eta}.
\end{align*}
In the above expressions, for $z=(\vartheta,h)$,
\begin{align*}
F_{\lambda,\mu,\eta}(z,z^{*}_{\lambda,\mu,\eta})&=(1+\xi^\prime \lambda)\Tlme F(\Tlmei(z+z^*_{\lambda,\mu,\eta}))\\
&= (\kappa_{A,\lambda,\mu,\eta}-\kappa(\vartheta+\vartheta^*_{\lambda,\mu,\eta}))[\partial_t(\vartheta+\vartheta^*_{\lambda,\mu,\eta})-B_{\lambda,\mu,\eta}(\vartheta+\vartheta^*_{\lambda,\mu,\eta})]\\
&\quad + (1+\xi^\prime \lambda)\{(d(\vartheta+\vartheta^*_{\lambda,\mu,\eta})-d_{A,\lambda,\mu,\eta})\Tue\Delta \Tuei(\vartheta+\vartheta^*_{\lambda,\mu,\eta})\\
&\quad -d(\vartheta+\vartheta^*_{\lambda,\mu,\eta})M_{2,\lambda,\mu,\eta}(h+h^*_{\lambda,\mu}):\Tue\nabla^2 \Tuei (\vartheta+\vartheta^*_{\lambda,\mu,\eta})\\
&\quad +d^\prime(\vartheta+\vartheta^*_{\lambda,\mu,\eta})|(I- M_{1,\lambda,\mu,\eta}(h +h^*_{\lambda,\mu}))\Tue\nabla \Tuei (\vartheta+\vartheta^*_{\lambda,\mu,\eta})|^2\\
&\quad -d(\vartheta+\vartheta^*_{\lambda,\mu,\eta})(M_{3,\lambda,\mu,\eta}(h +h^*_{\lambda,\mu})| \Tue\nabla \Tuei (\vartheta+\vartheta^*_{\lambda,\mu,\eta}))\\
&\quad + \kappa(\vartheta+\vartheta^*_{\lambda,\mu,\eta})\mathcal{R}_{\lambda,\mu,\eta}({h}+h^*_{\lambda,\mu})(\vartheta+\vartheta^*_{\lambda,\mu,\eta})\},
\end{align*}
with 
\begin{align*}
\begin{split}
\Upsilon_{\lambda,\mu,\eta}(h)&=\Tlme \Upsilon(\Tlmi h)=(\zeta\circ d_\M) (h\circ \Pi) (\nu_\M \circ\Pi)_{\lambda,\mu,\eta}\\
M_{1,\lambda,\mu,\eta}(h)&=\Tlme M_1(\Tlmi h)\\
&=[(I+ (\Tue\nabla\Tuei \Upsilon_{\lambda,\mu,\eta}(h))^{\sf T})^{-1} (\Tue\nabla\Tuei \Upsilon_{\lambda,\mu,\eta}(h))^{\sf T}]^{\sf T},\\
M_{2,\lambda,\mu,\eta}(h)&=\Tlme M_2(\Tlmi h)\\
&=M_{1,\lambda,\mu,\eta}(h)+(M_{1,\lambda,\mu,\eta}(h))^{\sf T} -M_{1,\lambda,\mu,\eta}(h)(M_{1,\lambda,\mu,\eta}(h))^{\sf T},\\
M_{3,\lambda,\mu,\eta}(h)&=\Tlme M_3(\Tlmi h)=(I-M_{1,\lambda,\mu,\eta}(h)): \Tue\nabla\Tuei M_{1,\lambda,\mu,\eta}(h),\\
\mathcal{R}_{\lambda,\mu,\eta}(h)\vartheta&=\Tlme \mathcal{R}(\Tlmi h)\Tlmei \vartheta\\
&=(\Tue\nabla\Tuei \vartheta |(I+(\Tue\nabla\Tuei \Upsilon_{\lambda,\mu,\eta}(h))^{\sf T})^{-1}\\
&\quad[\partial_t \Upsilon_{\lambda,\mu,\eta}(h)- B_{\lambda,\mu,\eta}(\Upsilon_{\lambda,\mu,\eta}(h))])/(1+\xi^\prime \lambda),
\end{split}
\end{align*}
and 
\begin{align*}
\bar{G}_{\lambda,\mu}(z,&z^*_{\lambda,\mu,\eta})=(1+\xi^\prime \lambda)\Tlm \bar{G}(\Tlmi(z+z^*_{\lambda,\mu,\eta}))\\
&=\{ \gamma(\vartheta+\vartheta^*_{\lambda,\mu})\beta_{\lambda,\mu}(h+h^*_{\lambda,\mu})-\gamma_{1,\lambda,\mu}\}[\partial_t({h}+h^*_{\lambda,\mu})-B_{\lambda,\mu}({h}+h^*_{\lambda,\mu})]\\
&\quad  +(1+\xi^\prime \lambda)\{-([\![\psi(\vartheta+\vartheta^*_{\lambda,\mu})]\!]+\sigma{\cH}_{\lambda,\mu}(h+h^*_{\lambda,\mu}))+l_{1,\lambda,\mu}(\vartheta+\vartheta^*_{\lambda,\mu})\\
&\quad  +\sigma_0 \Tu\Delta_\M\Tui ({h}+h^*_{\lambda,\mu}) -g_{1,\lambda,\mu}
\},
\end{align*}
with ${\cH}_{\lambda,\mu}=\Tlm \cH\Tlmi ,$ and
\begin{align*}
\begin{split}
M_{0,\lambda,\mu}(h)&=\Tlm M_0(\Tlmi h)=(I-h L_{\M,\lambda,\mu})^{-1},\hspace{1em}\\
\alpha_{\lambda,\mu}(h)&=\Tlm \alpha(\Tlmi h)=M_{0,\lambda,\mu}(h)\Tu\nabla\Tui h,\\
\beta_{\lambda,\mu}(h)&=\Tlm \beta(\Tlmi h)=(1+|\alpha_{\lambda,\mu}(h)|^2)^{-1/2},
\end{split}
\end{align*}
and 
\begin{align*}
\bar{Q}_{\lambda,\mu}&(z ,z^*_{\lambda,\mu,\eta})=(1+\xi^\prime \lambda)\Tlm \bar{Q}(\Tlmi(z+z^*_{\lambda,\mu,\eta}))\\
&=\gamma(\vartheta+\vartheta^*_{\lambda,\mu})\beta_{\lambda,\mu}(h+h^*_{\lambda,\mu})[\partial_t({h}+h^*_{\lambda,\mu})-B_{\lambda,\mu}({h}+h^*_{\lambda,\mu})]^2/(1+\xi^\prime \lambda)\\
&\quad+(l_{A,\lambda,\mu}-l(\vartheta+\vartheta^*_{\lambda,\mu}))[\partial_t({h}+h^*_{\lambda,\mu})-B_{\lambda,\mu}({h}+h^*_{\lambda,\mu})]\\
&\quad+(1+\xi^\prime \lambda)\{[\![(d(v+\bar{\vartheta}_{\lambda,\mu})-d_{A,\lambda,\mu})\partial_\nu (\vartheta+\vartheta^*_{\lambda,\mu})]\!]-q_{1,\lambda,\mu}\\
&\quad-( [\![d(\vartheta+\vartheta^*_{\lambda,\mu})\Tu\nabla \Tui(\vartheta+\vartheta^*_{\lambda,\mu})]\!]  | M_{4,\lambda,\mu}(h+h^*_{\lambda,\mu})\Tu\nabla_\M \Tui({h}+h^*_{\lambda,\mu})) \}
\end{align*}
with $M_{4,\lambda,\mu}(h)=\Tlm M_4(\Tlmi h)=(I -M_{1,\lambda,\mu}(h ))^{\sf T} M_{0,\lambda,\mu}(h)$.
\smallskip\\
Consider the map $\Phi:\zej\rightarrow \zfj$: $((\vartheta,h),(\lambda,\mu,\eta))\mapsto$
\begin{align*}
\begin{cases}
\kappa_{A,\lambda,\mu,\eta}\partial_t \vartheta -(1+\xi^\prime\lambda) d_{A,\lambda,\mu,\eta}\Tue\Delta\Tuei \vartheta-\kappa_{A,\lambda,\mu,\eta}B_{\lambda,\mu,\eta}(\vartheta)\hspace{3em} &\\
\quad\quad -F_{\lambda,\mu,\eta}(z ,z^*_{\lambda,\mu,\eta}) &\text{in}\hspace{.5em}  \Omega\!\setminus\!\M,\\
(1+{\rm sgn}(\gamma) \xi^{\prime}\lambda)(l_{1,\lambda,\mu} \vartheta +\sigma_0\Tu\Delta_\M\Tui h)-\gamma_{1,\lambda,\mu}\partial_t h \\
\quad\quad +\gamma_{1,\lambda,\mu} B_{\lambda,\mu}(h) - \bar{G}_{\lambda,\mu}(z,z^*_{\lambda,\mu,\eta}) &\text{on}\hspace{.5em} \M,\\
l_{A,\lambda,\mu}\partial_t h-(1+\xi^{\prime}\lambda)( [\![d_{A,\lambda,\mu} \partial_\nu \vartheta]\!]) -l_{A,\lambda,\mu}B_{\lambda,\mu}(h)-\bar{Q}_{\lambda,\mu}(z ,z^*_{\lambda,\mu,\eta}) &\text{on}\hspace{.5em} \M,
\end{cases}
\end{align*}
where $z=(\vartheta,h)$. Note that $\Phi(\bar{z}_{\lambda,\mu},(\lambda,\mu))=0$ for all $(\lambda,\mu)\in\Br$. It is understood that $\zfj:=\fa\times\zfb\times\zfc$.
\smallskip\\
(i) First, we shall check that $\Phi$ actually maps into $\zfj$. Functions in $\zfj$ automatically satisfy ($\mathcal{LCC}$). One can check the temporal traces without difficulty. Indeed, recalling that $z^*_{\lambda,\mu,\eta}(0,\cdot)=z^*(0,\cdot)$, we have
\begin{align*}
&\bar{G}_{\lambda,\mu}(z,z^*_{\lambda,\mu,\eta})(0)=\Tlm G(z^*)(0)-g_1(0)=0,\\
&\bar{Q}_{\lambda,\mu}(z,z^*_{\lambda,\mu,\eta})(0)=\Tlm Q(z^*)(0)-q_1(0)=0.
\end{align*} 
It suffices to check regularity of $F_{\lambda,\mu,\eta}(z ,z^*_{\lambda,\mu,\eta})$, $\bar{G}_{\lambda,\mu}(z,z^*_{\lambda,\mu,\eta})$ and $\bar{Q}_{\lambda,\mu}(z,z^*_{\lambda,\mu,\eta})$, which will become clear in our argument for the regularity of $\Phi$.
\smallskip\\
It follows from  well-known results for substitution operators for Sobolev spaces and \eqref{S5: ana of z*} that
\begin{align*}
[(\vartheta,(\lambda,\mu,\eta))\mapsto \kappa(\vartheta+\vartheta^*_{\lambda,\mu,\eta})]\in C^k (\zea\times\Br, C(J, BC(\Omega))).
\end{align*}
Analogous statements hold for $d(\vartheta+\vartheta^*_{\lambda,\mu,\eta})$, $d^\prime(\vartheta+\vartheta^*_{\lambda,\mu,\eta})$, $l(\vartheta+\vartheta^*_{\lambda,\mu,\eta})$, $\gamma(\vartheta+\vartheta^*_{\lambda,\mu,\eta})$ and $\psi(\vartheta+\vartheta^*_{\lambda,\mu,\eta})$ as well.
\smallskip\\
Adopting the notation in Section~2, we introduce an extension operator 
$$e_\Pi:\eb\rightarrow \mathbb{E}_2(J,\T_a):\quad h\mapsto h\circ\Pi.$$ 
$((\pi_\kappa\circ\Pi)^2)_{\kappa\in\K}$ forms a partition of unity for $\T_a$. For $\F\in\{W_p,H_p\}$, we define
\begin{align*}
\begin{split}
&\mathcal{R}_\Pi^c: \F^s(\T_a, E)\rightarrow \prod\limits_{\kappa\in\K}\F^s(\T_{a,\kappa},E):u\mapsto (u\pi_\kappa\circ\Pi )_\kappa,\\
&\mathcal{R}_\Pi:\prod\limits_{\kappa\in\K}\F^s(\T_{a,\kappa},E)\rightarrow \F^s(\T_a, E): (u_\kappa)_\kappa\mapsto \sum\limits_{\kappa\in\K} u_\kappa\pi_\kappa\circ\Pi .
\end{split}
\end{align*} 
Then $\mathcal{R}_\Pi$ is a retraction with $\mathcal{R}_\Pi^c$ as a coretraction. 
By using this retraction-coretraction system, it is a simple matter to check that
\begin{align*}
e_\Pi\in \L(\eb,\mathbb{E}_2(J,\T_a)).
\end{align*}
Extending $\mathcal{V}:=(\zeta\circ d_\M)(\nu_\M\circ \Pi)$ to be identically zero outside $\T_a$, it belongs to $BC^\infty(\Omega,\R^{m+1})\cap C^\omega(\T_{a/3},\R^{m+1})$. 
Proposition~\ref{S3: Omega-Tue}(a) implies
$$[(\lambda,\mu,\eta)\mapsto \mathcal{V}_{\lambda,\mu,\eta}]\in C^\omega(\Br, C^1(J, BC^k(\Omega,\R^{m+1})))$$
for all $k\in\N_0$. Proposition~\ref{S3：M-reg} and \eqref{S5: ana of z*} then yield 
\begin{align*}
[(h,(\lambda,\mu))\mapsto \Upsilon_{\lambda,\mu,\eta}({h}+h^*_{\lambda,\mu})]\in C^\omega(\zeb\times\Br, \mathbb{E}_2(J,\Omega;\R^{m+1})).
\end{align*} 
Because $\mathbb{F}_j(J,\Omega)$ with $j=2,3$ and $W^{1-1/2p}_p(J;H^1_p(\Omega))\cap L_p(J; W^{3-1/p}_p(\Omega))$ are Banach algebras, we verify via Proposition~\ref{S3: Omega-Tue}(b) that
\begin{align}
\label{S5: ana of M_j}
[(h,(\lambda,\mu,\eta))\mapsto M_{j,\lambda,\mu,\eta}({h}+h^*_{\lambda,\mu})]\in C^\omega(\zeb\times\Br, \mathbb{F}_2(J,\Omega;\R^{(m+1)^2}))
\end{align}
for $j=1,2,3$. Proposition~\ref{S3: Omega-Tue}(b) and \eqref{S5: ana of z*} lead to
$$[(\vartheta,(\lambda,\mu,\eta)) \mapsto\Tue\nabla\Tuei (\vartheta+\vartheta_{\lambda,\mu,\eta})]\in C^\omega(\zea\times\Br, \mathbb{F}_3(J,\Omega_i;\R^{m+1})). $$
In virtue of Proposition~\ref{S3: Omega-Tue}(b) and \eqref{S3: analyticity of Blme}, we get
\begin{align*}
[(z,(\lambda,\mu,\eta))\mapsto \mathcal{R}_{\lambda,\mu,\eta}(h+h^*_{\lambda,\mu})(\vartheta+\vartheta^*_{\lambda,\mu,\eta})]\in C^\omega(\zej\Br, \mathbb{F}_3(J,\Omega_i)).
\end{align*}
Now it is immediate that 
$$[(z,(\lambda,\mu))\mapsto F_{\lambda,\mu,\eta}(z,z^*_{\lambda,\mu})]\in C^k(\zej\times\Br, \fa).$$
Since $\cb$ is a multiplication algebra, we have
\begin{align*}
[(h,(\lambda,\mu))\mapsto M_{0,\lambda,\mu}(h+h^*_{\lambda,\mu})]\in C^\omega(\zeb\times\Br, \mathbb{C}_2(J;T^1_1\M)).
\end{align*}
Proposition~\ref{S3：M-reg} and point-wise multiplier results on $\M$ then imply that
\begin{align*}
[(h,(\lambda,\mu))\mapsto \alpha_{\lambda,\mu}(h+h^*_{\lambda,\mu})]\in C^\omega(\zeb\times\Br, \mathbb{F}_2(J;T\M)).
\end{align*}
Combined with  the properties for substitution operators for Sobolev-Slobodeckii spaces, it yields
\begin{align*}
[(h,(\lambda,\mu))\mapsto \beta_{\lambda,\mu}(h+h^*_{\lambda,\mu})]\in C^\omega(\zeb\times\Br, \fb).
\end{align*} 
It is shown in  \cite{Shao13} that in every coordinate patch $(\Ok,\vpk)$, the local expression of the mean curvature operator reads as
\begin{align*}
({\cH}(h))_\kappa=\beta(h) \frac{P_\kappa(h,\partial_j h, \partial_{jk}h)}{R_\kappa(h,\partial_j h)}.
\end{align*}
Here $P_\kappa$ is a polynomial in $h$ and its derivatives up to second order, and $R_\kappa$ is a polynomial in $h$ and its first order derivatives. Both have $BC^\infty \cap C^\omega$-coefficients. We again use the fact that $\fb$ is a Banach algebra. Following a similar argument in the same reference, we get 
\begin{align*}
[(h,(\lambda,\mu))\mapsto {\cH}_{\lambda,\mu}(h+h^*_{\lambda,\mu})]\in C^\omega(\zeb\times\Br, \fb).
\end{align*}
By \eqref{S5: ana of M_j} and trace theorems of anisotropic Sobolev-Slobodeckii spaces, we infer that
\begin{align*}
[(h,(\lambda,\mu))\mapsto M_{1,\lambda,\mu}(h+h^*_{\lambda,\mu})|_\M]\in C^\omega(\zeb\times\Br, \mathbb{F}_3(J;\R^{(m+1)^2}).
\end{align*}
It follows from the point-wise multiplier theorem in \cite[Section~9]{Ama13} that
\begin{align*}
[(h,(\lambda,\mu))\mapsto M_{4,\lambda,\mu}(h+h^*_{\lambda,\mu})]\in C^\omega(\zeb\times\Br, \mathbb{F}_3(J;\R^{(m+1)^2}).
\end{align*}
To sum up, we conclude that
\begin{align*}
\begin{split}
[(z,(\lambda,\mu))\mapsto \bar{G}_{\lambda,\mu}(z,z^*_{\lambda,\mu})]&\in C^k(\zej\times\Br, \fb),\\
[(z,(\lambda,\mu))\mapsto \bar{Q}_{\lambda,\mu}(z,z^*_{\lambda,\mu})]&\in C^k(\zej\times\Br, \fc).
\end{split}
\end{align*}
Taking into account all the above discussion, this leads to
\begin{align*}
\Phi\in C^k(\zej\times\Br, \zfj).
\end{align*}
\smallskip\\
(ii) We look at the Fr\'echet derivative of $\Phi$ with respect to $\zej$ at $(\bar{z},0)$. To this end, we find it more convenient to consider the following non-linear maps.
\begin{align*}
\hat{F}(z)=&(\hat{\kappa}_0-\kappa(\vartheta))\partial_t \vartheta +(d(\vartheta)-\hat{d}_0)\Delta \vartheta -d(\vartheta)M_2(h):\nabla^2 \vartheta \\
&+d^\prime(\vartheta)|(I-M_1(h))\nabla \vartheta|^2 -d(\vartheta)(M_3(h)| \nabla \vartheta) +\kappa(\vartheta)\mathcal{R}(h)\vartheta,\\
 \hat{G}(z)=& -([\![\psi(\vartheta)]\!] + \sigma {\cH}(h) ) +\hat{l}_1 \vartheta +\sigma_0\Delta_\M h +(\gamma(\vartheta)\beta(h)-\hat{\gamma}_1)\partial_t h,\\
 \hat{Q}(z)=&[\![ (d(\vartheta)-\hat{d}_0)\partial_\nu \vartheta ]\!] +(\hat{l}_0 -l(\vartheta))\partial_t h -([\![d(\vartheta)\nabla \vartheta]\!]| M_4(h)\nabla_\M h)\\
 &+\gamma(v)\beta(h)(\partial_t h)^2,
\end{align*}
i.e., replacing  $\vartheta_A$ by $\hat{\vartheta}_0$ in the definition of $F$, $G$ and $H$. With only slight modification of (i), we immediately obtain
\begin{align*}
(\hat{F},\hat{G},\hat{Q})\in C^k(\ej,\fa\times\fb\times\fc).
\end{align*}
Recall that $\hat{z}=\bar{z}+z^*$. Letting $w=(u,\rho)$, one computes 
\begin{align*}
D_1\Phi(\bar{z},0)w=
\left\{\begin{aligned}
&\hat{\kappa}_0 \partial_t u -\hat{d}_0 \Delta u -\hat{F}^\prime(\hat{z})w && \text{in} && \Omega\setminus\M, \\
&\hat{l}_1 u +\sigma_0 \Delta_\M \rho -\hat{\gamma}_1\partial_t\rho - \hat{G}^\prime(\hat{z})w   && \text{on}&& \M,\\
&\hat{l}_0  \partial_t\rho -[\![ \hat{d}_0 \partial_\nu u]\!] -\hat{Q}^\prime(\hat{z})w  && \text{on}&& \M.
\end{aligned}\right.
\end{align*}
A moment of reflection shows that we have the liberty to exchange the coefficients $(\kappa_A,d_A,l_A)$, 
used in the definition of $\Phi$, by $(\hat{\kappa}_0,\hat{d}_0, \hat{l}_0)$.
These quantities are only used for the principal linearization, and are then subtracted off in the nonlinearities.

We split $D_1\Phi(\bar{z},0)$ into two parts. 
Let $\bz:=\beta(\hat{h}_0)$.
Define $\mathbb{L}(\hat{\vartheta}_0):\zej\rightarrow \zfj$ by
\begin{align*}
\mathbb{L}(\hat{\vartheta}_0)w=
\left\{\begin{aligned}
&\hat{\kappa}_0 \partial_t u -\hat{d}_0 (I-M_2(\hat{h}_0)): \nabla^2 u && \text{in}&& \Omega\setminus\M, \\
&\hat{l}_1 u +\sigma \cH^\prime(\hat{h}_0) \rho -\hat{\gamma}_1\beta^\prime(\hat{h}_0)e^{\Delta_\M t}(\partial_t \hat{h}(0))\rho\\ 
&\quad  -\hat{\gamma}_1 \bz\partial_t\rho-\gamma^\prime(e^{\Delta_\M t}\hat{\vartheta}_0)\bz e^{\Delta_\M t}(\partial_t \hat{h}(0))u    && \text{on} && \M,\\
&\hat{l}_0  \partial_t\rho -([\![ \hat{d}_0 \nabla u]\!]| \nu_\M- M_4(\hat{h}_0)\nabla_\M \hat{h}_0)    && \text{on}&& \M. 
\end{aligned}\right.
\end{align*}
and $\mathbb{K}(\hat{z}):\zej\rightarrow \zfj$ by 
$$\mathbb{K}(\hat{z})=:(\mathbb{K}_1,\mathbb{K}_2,\mathbb{K}_3)^{\sf T}:=\mathbb{L}(\hat{\vartheta}_0)w-D_1\Phi(\hat{z},0).$$
\begin{theorem}
Let $J=[0,T]$. Then $\mathbb{L}(\hat{\vartheta}_0)\in \Lis(\ej,\fj)$. In particular, $\mathbb{L}(\hat{\vartheta}_0)\in \Lis(\zej,\zfj)$. 
\end{theorem}
\begin{proof}
We first solve the following model problem for the case $\gamma\equiv 0$. Suppose that 
$$\kappa_0 \in BU\!C(\R^m \times \R_+),\quad d_0\in BU\!C^1(\R^m),\quad a_0\in BU\!C^1(\R^m;\R^m) ,\quad \kappa_0,d_0>0. $$ 
Moreover, $l_0\in W^{2-6/p}_p(\R^m)$, $l_2\in \mathbb{F}_2(J,\R^m)$ satisfy $l_2 l_0>0$. The differential operators $-P(x,y):\nabla^2$ and $-S(x):\nabla^2_x$ are uniformly elliptic. For any given $(f,g,q,(u_0,\rho_0))\in \mathbb{F}(J,\R^m \times \R_+)$,
\begin{align*}
\left\{\begin{aligned}
\kappa_0(x,y) \partial_t u(x,y) -  P(x,y): \nabla^2 u(x,y) &=f(x,y)
 &&\text{in}&& \R^m \times \R_+ ,\\
l_2(t,x) u(x) + S:\nabla^2_x \rho(x) &=g(x)
 &&\text{on}&&\R^m,\\
l_0(x) \partial_t \rho(x) + d_0(x)\partial_\nu u(x) +a_0(x)\cdot\nabla_x u(x) &=q(x)
 &&\text{on}&&\R^m,\\
u(0)=u_0,\quad \rho(0)&=\rho_0,&&
\end{aligned}\right.
\end{align*}
admits a unique solution $(u,\rho)\in \mathbb{E}(J,\R^m \times \R_+)$. We identify $\R^m$ with $\R^m\times\{0\}$. Here $x\in\R^m,y\in\R_+$,  and the space $\mathbb{F}(J,\R^m \times \R_+)$ is defined by replacing $\Omega\setminus\M$ and $\M$ by $\R^m \times \R_+$ and $\R^m$, respectively. The space $\mathbb{E}(J,\R^m \times \R_+)$ is defined analogously. Similar problems have been considered in \cite{DenPruZac08}. It suffices to check the Lopatinskii-Shapiro condition (${\bf LS}$) and the asymptotic Lopatinskii-Shapiro condition (${\bf LS}^+_\infty$) defined therein. For simplicity, we assume all the coefficients to be constants and set
$$P(\xi):=\xi_e^{\sf T}P\xi_e,\quad S(\xi):=\xi^{\sf T}S\xi, \quad P_{m+1}:=P_{m+1,m+1} $$
and $P_m:=(P_{m+1, 1},\cdots,P_{m+1,m})$ with $\xi\in \R^m$ and $\xi_e:=(\xi^{\sf T},0)^{\sf T}\in\R^{m+1}$.
(${\bf LS}$) is satisfied, if for any $\xi\in\R^m$ and $\lambda\in \bar{\C}_+$  with $|\xi|+|\lambda|\neq 0$, the following ordinary differential equation in $\R_+$
\begin{align*}
\left\{\begin{aligned}
(\kappa_0\lambda +P(\xi)+2iP_m\cdot\xi\partial_y -  P_{m+1}\partial^2_y) v(y) &=0,
 && y>0 ,\\
l_2 v(0) - S(\xi) \delta &=0, \\
l_0 \lambda \delta - (i a_0\cdot\xi +d_0(x)\partial_y ) v(0) &=0,\\
\end{aligned}\right.
\end{align*}
has a unique solution $(v,\delta)\in C_0(\R_+;\C)\times \C$.
It is clear that the only stable solution to the first line is 
$$\displaystyle v(y)=e^{\mu y}v(0), \quad  \mu:=\frac{\sqrt{-(P_m\cdot\xi)^2+P_{m+1}\kappa_0\lambda+P_{m+1}P(\xi)}+iP_m\cdot\xi}{ P_{m+1}}$$
in the case $\lambda\notin \R$, or, $v\equiv 0$, otherwise. We conclude from the second and third lines that 
$$[i S(\xi)a_0\cdot\xi  +S(\xi)\mu d_0  -l_2 l_0 \lambda]\delta=0.$$
It implies that $\delta=0$, $v\equiv 0$. (${\bf LS}^+_\infty$) is satisfied, if
for any $\xi\in\R^m$ and $\lambda\in \bar{\C}_+$  with $|\xi|+|\lambda|\neq 0$,
\begin{align*}
\left\{\begin{aligned}
(\kappa_0\lambda +P(\xi)+2iP_m\cdot\xi\partial_y -  P_{m+1}\partial^2_y) v(y) &=0,
 && y>0 ,\\
l_2 v(0) - S(\xi) \delta &=0, \\
(i a_0\cdot\xi +d_0(x)\partial_y ) v(0) &=0,\\
\end{aligned}\right.
\end{align*}
and $|\xi|=1$, $\lambda\in\bar{\C}_+\setminus\{0\}$
\begin{align*}
\left\{\begin{aligned}
(\kappa_0\lambda -  P_{m+1}\partial^2_y) v(y) &=0,
 && y>0 ,\\
l_2 v(0) - S(\xi) \delta &=0, \\
l_0 \lambda \delta +d_0(x)\partial_y  v(0) &=0,\\
\end{aligned}\right.
\end{align*}
and $|\xi|=1$, $\lambda\in\bar{\C}_+\setminus\{0\}$
\begin{align*}
\left\{\begin{aligned}
(\kappa_0\lambda -  P_{m+1}\partial^2_y) v(y) &=0,
 && y>0 ,\\
l_2 v(0)- S(\xi) \delta &=0, \\
d_0(x)\partial_y  v(0) &=0,\\
\end{aligned}\right.
\end{align*}
admit unique solutions $(v,\delta)\in C_0(\R_+;\C)\times \C$.
One can check in an analogous manner that stable solutions to those equations are trivial.   
For general coefficients, the problem can be proved  by a perturbation argument as in  \cite{DenHiePru03}.
The rest of the proof now follows from similar arguments to those of \cite[Theorems~3.3, 3.5]{PruSimZac12}. 
\end{proof}
In order to prove that $D_1\Phi(\bar{z},0)$ is an isomorphism, we need to control the norm $\| \mathbb{K}(\hat{z}_0)\|_{\L(\zej,\zfj)}$. To this end, we first compute several derivatives related to $\mathbb{K}$ explicitly.
\begin{align*}
\hat{F}^\prime(z)w=&(\hat{\kappa}_0- \kappa(\vartheta))\partial_t u -\kappa^\prime(\vartheta)u \partial_t \vartheta +(d(\vartheta)-\hat{d}_0)\Delta u +d^\prime(\vartheta)u\Delta \vartheta\\
& -d^\prime(\vartheta)u M_2(h):\nabla^2 \vartheta -d(\vartheta)M_2^\prime(h)\rho:\nabla^2 \vartheta
-d(\vartheta)M_2(h):\nabla^2 u\\
& -d^\prime(\vartheta)u (M_3(h)|\nabla \vartheta) -d(\vartheta)(M_3^\prime(h)\rho | \nabla \vartheta)-d(\vartheta)(M_3(h)| \nabla u)\\
& +2d^\prime(\vartheta)((I-M_1(h)\nabla \vartheta|(I-M_1(h))\nabla u -M_1^\prime(h)\rho\nabla \vartheta)\\
&  +d^{\prime\prime}(\vartheta)u |(I-M_1(h))\nabla \vartheta|^2 
+\kappa(\vartheta)(\nabla u| (I+\nabla\Upsilon(h)^{\sf T})^{-1}\partial_t\Upsilon(h)) \\
& -\kappa(\vartheta)(\nabla \vartheta| (I+\nabla\Upsilon(h)^{\sf T})^{-1}\nabla\Upsilon(\rho)^{\sf T}(I+\nabla\Upsilon(h)^{\sf T})^{-1}\partial_t\Upsilon(h))\\
&+\kappa(\vartheta)(\nabla \vartheta| (I+\nabla\Upsilon(h)^{\sf T})^{-1}\partial_t\Upsilon(\rho)) + \kappa^\prime(\vartheta)u \mathcal{R}(h)\vartheta ,
\end{align*}
and
\begin{align*}
\hat{G}^\prime(z)w=& -([\![\psi^\prime(\vartheta)]\!]u+\sigma {\cH}^\prime(h)\rho) +\hat{l}_1 u +\sigma_0\Delta_\M \rho +(\gamma(\vartheta)\beta(h)-\hat{\gamma}_1)\partial_t\rho \\
& +(\gamma^\prime(\vartheta)u\beta(h) -\gamma(\vartheta)\beta^\prime(h)\rho)\partial_t h,
\end{align*}
and
\begin{align*}
\hat{Q}^\prime(z)w=& [\![d^\prime(\vartheta)\partial_\nu \vartheta ]\!]u 
+[\![(d(\vartheta)-\hat{d}_0)\partial_\nu u ]\!] 
+(\hat{l}_0-l(\vartheta))\partial_t \rho 
-l^\prime(\vartheta)u \partial_t h\\
&-([\![d^\prime(\vartheta)\nabla \vartheta]\!]u| M_4(h)\nabla_\M h) 
-([\![d(\vartheta)\nabla u ]\!]| M_4(h)\nabla_\M h) \\
&-([\![d(\vartheta)\nabla \vartheta ]\!]| M_4(h)\nabla_\M \rho) 
-([\![d(\vartheta)\nabla \vartheta ]\!]| M_4^\prime(h)\rho \nabla_\M h) \\
& +\gamma^\prime(\vartheta)u \beta(h)[\partial_t h]^2 
+\gamma(\vartheta) \beta^\prime(h)\rho [\partial_t h]^2
+2\gamma(\vartheta)\beta(h) \partial_t h \partial_t \rho.
\end{align*}
The derivatives of $M_0(h)$, $\alpha(h)$, $\beta(h)$ and ${\cH}(h)$ are given by
\begin{align*}
& M_0^\prime(h)\rho=\rho M_0(h)L_\M M_0(h),\hspace{1em} \alpha^\prime(h)\rho=M_0(h)\nabla_\M \rho +\rho M_0(h)L_\M M_0(h)\alpha(h),\\
& \beta^\prime(h)\rho=-\beta^3(h)(\alpha(h)|M_0(h)\nabla_\M \rho +\rho M_0(h)L_\M\alpha(h)),\\
& {\cH}^\prime(h)\rho=\beta(h)\{ {\tr}[M_0^\prime(h)\rho(L_\M +\nabla_\M \alpha(h))] +{\tr}[M_0(h)\nabla_\M \alpha^\prime(h)\rho]\\
&\quad -2\beta(h)\beta^\prime(h)\rho(M_0(h)\alpha(h)| [\nabla_\M\alpha(h)]\alpha(h) ) -\beta^2(h)(M_0^\prime(h)\rho\alpha(h)|[\nabla_\M \alpha(h)]\alpha(h) )\\
&\quad -\beta^2(h)(M_0(h)\alpha^\prime(h)\rho| [\nabla_\M \alpha(h)]\alpha(h) )
-\beta^2(h)(M_0(h)\alpha(h)| [\nabla_\M \alpha^\prime(h)\rho]\alpha(h) )\\
&\quad -\beta^2(h)(M_0(h)\alpha(h)| [\nabla_\M\alpha(h)]\alpha^\prime(h)\rho) \}/m +\beta^\prime(h)\rho ({\cH}(h)/\beta(h)).
\end{align*}
See \cite[formula~(32)]{PruSim13} for a justification for the last equality.

We will use the following lemma frequently in the sequel.
\begin{lem}
\label{S5: F2-lem}
There exists a constant $C_0$ independent of $T$ such that
\begin{itemize}
\item[(a)] For all $(v_1,v_2)\in \mathbb{F}_j(J)\times \prescript{}{0}{\mathbb{F}_j(J)}$ and  $j=2,3$,
\begin{align*}
\|v_1 v_2\|_{\mathbb{F}_j(J)}\leq C_0(\|v_1\|_{C(J\times\M)}+\|v_1\|_{\mathbb{F}_j(J)})\|v_2\|_{\mathbb{F}_j(J)}.
\end{align*}
\vspace{-6mm}
\item[(b)] For all $(v_1,v_2)\in \prescript{}{0}{\mathbb{F}_j(J)}\times \prescript{}{0}{\mathbb{F}_j(J)}$ and  $j=2,3$,
\begin{align*}
\|v_1 v_2\|_{\prescript{}{0}{\mathbb{F}_j(J)}}\leq C_0  \|v_1\|_{\prescript{}{0}{\mathbb{F}_j(J)}}\|v_2\|_{\prescript{}{0}{\mathbb{F}_j(J)}}.
\end{align*}
\vspace{-6mm}
\item[(c)] For all $(v_1,v_2)\in \fc\times \zfb$
\begin{align*}
\|v_1 v_2\|_{\fc}\leq C_0  \|v_1\|_{\fc} \|v_2\|_{\zfb}.
\end{align*}
\end{itemize}
\end{lem}
\begin{proof}
(a) The case $j=2$ is shown in \cite[Lemma~5.5(b)]{PruSim09}. 
For the reader's convenience, we will nevertheless include a proof herein. The case $j=3$ follows in a similar way. 
By the retraction-coretraction system defined in Section~2, it suffices to show the estimates for functions on $\Q$, i.e., we assume $(v_1,v_2)\in \mathbb{F}_j(J,\Q)\times \prescript{}{0}{\mathbb{F}_j(J,\Q)}$.
We equip $\mathbb{F}_2(J,\Q)$ with the norm:
\begin{align*}
\begin{split}
&\|v\|_{\mathbb{F}_2(J,\Q)}= \|v\|_{W^{1-1/2p}_p(J;L_p(\Q))} +\|v\|_{L_p(J;W^{2-1/p}_p(\Q))},\\
&\|v\|_{L_p(J;W^{2-1/p}_p(\Q))}=\|v\|_{L_p(J;H^1_p(\Q))} +\sum\limits_{j=1}^m \| \langle \partial_j v \rangle_{W^{1-1/p}_p(\Q)}\|_{L_p(J)}. 
\end{split}
\end{align*}
Here $\langle \cdot \rangle_{W^{1-1/p}_p(J;X)}$ is the Slobodeckii seminorm of the space $W^{1-1/p}_p(J;X)$ for a Banach space $X$.
Then,
\begin{align*}
&\|v_1 v_2\|_{W^{1-1/2p}_p(J;L_p(\Q))}\leq C^\prime \{ \|v_1\|_{W^{1-1/2p}_p(J;L_p(\Q))}\|v_2\|_{BC(J\times\Q)}\\
&\quad + (\int_J\int_J \|v_1(s)(v_2(t)- v_2(s))\|^p_{L_p(\Q)}\frac{1}{|t-s|^{1/2+p}}\, dt\, ds  )^{1/p} \} \\
&\leq C^\prime\{ \|v_1\|_{W^{1-1/2p}_p(J;L_p(\Q))}\|v_2\|_{BC(J\times\Q)}+\|v_1\|_{BC(J\times\Q)}\langle v_2 \rangle_{W^{1-1/2p}_p(J;L_p(\Q))} \}.
\end{align*}
Similarly, one computes
\begin{align*}
&\quad \|v_1 v_2\|_{L_p(J;W^{2-1/p}_p(\Q))}\leq C^\prime\{ \|v_1\|_{L_p(J;W^{2-1/p}_p(\Q))}\|v_2\|_{C(J,BC^1(\Q))}\\
&+\sum\limits_{j=1}^m[\int_J (\int_{\Q}\!\int_{\Q} |v_1(t,x)(\partial_j v_2(t,x)-\partial_j v_2(t,y) |^p \frac{1}{|x-y|^{m-1+p}}\,dx\, dy)\, dt ]^{1/p}\\
&+\sum\limits_{j=1}^m[\int_J (\int_{\Q}\!\int_{\Q} |\partial_j v_1(t,x)(v_2(t,x)-v_2(t,y) |^p \frac{1}{|x-y|^{m-1+p}}\,dx\, dy)\, dt ]^{1/p}
\}.
\end{align*}
We immediately have
\begin{align*}
&\quad\int_J (\int_{\Q}\!\int_{\Q} |v_1(t,x)(\partial_j v_2(t,x)-\partial_j v_2(t,y) |^p \frac{1}{|x-y|^{m-1+p}}\,dx\, dy)\, dt\\
&\leq \|v_1\|_{BC(J\times\Q)}^p\| v_2\|_{L_p(J,W^{2-1/p}_p(\Q))}^p.
\end{align*}
The remaining estimate can be carried out as follows.
\begin{align}
&\notag \quad \int_{\Q}\!\int_{\Q} |\partial_j v_1(t,x)(v_2(t,x)-v_2(t,y) |^p \frac{1}{|x-y|^{m-1+p}}\,dx\, dy\\
&\notag \leq \int_{\Q}\!\int_{\Q} |\partial_j v_1(t,x)|^p (\int\limits_0^1 |(\nabla v_2(t,x+\tau(y-x))|(y-x))|\, d\tau)^p\frac{1}{|x-y|^{m-1+p}}\,dx\, dy\\
&\notag\leq C^\prime\int_{\Q}|\partial_j v_1(t,x)|^p \int_{\Q} \frac{|x-y|^p}{|x-y|^{m-1+p}}\,dy \, dx \|\nabla v_2(t)\|_{BC(\Q)}^p\\
\label{S5: est 2}
&\leq C^\prime\|v_1(t)\|_{H^1_p(\Q)}^p\|  v_2(t)\|_{BC^1(\Q)}^p.
\end{align}
Combining these discussions yields
\begin{align*}
\|v_1 v_2\|_{L_p(J;W^{2-1/p}_p(\Q))}&\leq
C^\prime(\|v_1\|_{BC(J\times\Q)}+\|v_1\|_{L_p(J;W^{2-1/p}_p(\Q))})\\
&\quad(\|v_2\|_{C(J;BC^1(\Q))}+\|v_2\|_{L_p(J;W^{2-1/p}_p(\Q)})).
\end{align*}
Using the fact that the embedding constant of $\prescript{}{0}{\mathbb{F}_2(J,\Q)}\hookrightarrow C(J;BC^1(\Q))$ is independent of $T$ yields the asserted result.  
\smallskip\\
(b) is an immediate consequence of (a) and the fact that the embedding constant of 
$\prescript{}{0}{\mathbb{F}_2(J,\Q)}\hookrightarrow C(J;BC(\Q))$ is independent of $T$. 
\smallskip\\
(c) Suppose that $(v_1,v_2)\in \mathbb{F}_3(J,\Q)\times \prescript{}{0}{\mathbb{F}_2(J,\Q)}$. 
Then
\begin{align*}
&\|v_1 v_2\|_{W^{1/2-1/2p}_p(J;L_p(\Q))}\leq \|v_1\|_{W^{1/2-1/2p}_p(J;L_p(\Q))}\|v_2\|_{BC(J\times\Q)}\\
&\quad + (\int_J\int_J \|v_1(s)(v_2(t)- v_2(s))\|^p_{L_p(\Q)}\frac{1}{|t-s|^{1/2+p/2}}\, dt\, ds  )^{1/p}.
\end{align*}
By \cite[Lemma~2.5]{MeySch12}, the comments below \cite[formula~(3.5)]{MeySch12}, an analogue of \cite[diagram~(6.12)]{PruSim07} with $BU\!C$ replaced by $C^s$, we obtain the embedding result 
$$\prescript{}{0}{\mathbb{F}_2(J,\Q)}\hookrightarrow C^s(J;BC(\Q))$$ 
for some $s>1/2-1/2p$ with an embedding constant $M=M(s)$ uniform in $T$. 
Thus, we infer that
\begin{align*}
&\quad \int_J\int_J \|v_1(s)(v_2(t)- v_2(s))\|^p_{L_p(\Q)}\frac{1}{|t-s|^{1/2+p/2}}\, dt\, ds\\
&\leq C\int_J \|v_1(s)\|^p_{L_p(\Q)} \int_J \frac{1}{|t-s|^{1/2+p/2-sp}}\, dt\, ds \|v_2\|^p_{C^s(J;BC(\Q))}\\
&\leq C^\prime \|v_1(s)\|^p_{L_p(J;L_p(\Q))} \|v_2\|_{C^s(J;BC(\Q))}^p.
\end{align*}
By an analogous estimate as in (a), see in particular \eqref{S5: est 2}, one obtains
$$ \|v_1 v_2\|_{L_p(J;W^{1-1/p}_p(\Q))}\leq C^\prime \|v_1\|_{L_p(J;W^{1-1/p}_p(\Q))} \|v_2\|_{C(J;BC^1(\Q))}.$$
\end{proof}
Note that the multiplication constant $C_0$ in (b) blows up as $T\to 0$ if $\prescript{}{0}{\mathbb{F}_j(J)}$ is replaced by $\mathbb{F}_j(J)$.

We write
\begin{align*}
\mathbb{K}_2(\hat{z})w =& \hat{l}_1 u- [\![\psi^\prime(\hat{\vartheta})]\!]u +(\gamma(\hat{\vartheta})-\hat{\gamma}_1)\beta(\hat{h})\partial_t\rho +\hat{\gamma}_1(\beta(\hat{h})-\bz)\partial_t\rho\\
&  -\sigma( {\cH}^\prime(\hat{h})- {\cH}^\prime(\hat{h}_0))\rho  +(\gamma(\hat{\vartheta})\beta^\prime(\hat{h})\partial_t \hat{h} - \hat{\gamma}_1\beta^\prime(\hat{h}_0)e^{\Delta_\M t}(\partial_t \hat{h}(0)))\rho\\
& +(\gamma^\prime(\hat{\vartheta})\beta(\hat{h})\partial_t \hat{h}-\gamma^\prime(e^{\Delta_\M t}\hat{\vartheta}_0)\bz e^{\Delta_\M t}(\partial_t \hat{h}(0)))u.
\end{align*} 
Let $\varepsilon$ sufficiently small be fixed. Recall that
$$\hat{l}_1 u-[\![\psi^\prime(\hat{\vartheta})]\!]u=([\![\psi^\prime(e^{\Delta_\M t}\hat{\vartheta}_0)-\psi^\prime(\hat{\vartheta})]\!])u.$$
Due to fact that $[\![\psi^\prime(e^{\Delta_\M t}\hat{\vartheta}_0)-\psi^\prime(\hat{\vartheta})]\!]\in \zfb$ 
and Lemma~\ref{S5: F2-lem}(b), by making $T$ small enough, we can achieve that
$$
\|\hat{l}_1 u-[\![\psi^\prime(\hat{\vartheta})]\!]u\|_{\fb}\leq \varepsilon \|u\|_{\fb}. 
$$
The arguments for the remaining terms in $\mathbb{K}_2(\hat{z})w$ are similar. Thus given any $\varepsilon>0$, for $T$ small enough, we have
$$\|\mathbb{K}_2(\hat{z}) w\|_{\fb}\leq \varepsilon\|w\|_{\zej}. $$

One can obtain an analogous assertion for  
$\|\mathbb{K}_3(\hat{z})\|_{\L(\zej,\zfc)}$. Indeed, 
$$\| [\![ (d(\vartheta)-\hat{d}_0) \partial_\nu u ]\!]\|_{\fc} \leq C_0 \|(d(\vartheta)-\hat{d}_0)\|_{\zfc} \|[\![\partial_\nu u]\!]\|_{\zfc}. $$
Similar estimates also apply to 
$$\| (\hat{l}_0- l(\hat{\vartheta})) \partial_t \rho\|_{\fc}, \|( [\![ d(\hat{h})\nabla u ]\!]| M_4(\hat{h})\nabla \hat{h} ) -( [\![ \hat{d}_0 \nabla u ]\!]| M_4(\hat{h}_0)\nabla \hat{h}_0 )\|_{\fc}. $$
It follows from Lemma~\ref{S5: F2-lem}(c) that
$$ \|[\![ d^\prime(\hat{\vartheta}) \partial_\nu \hat{\vartheta} ]\!] u\|_{\fc} \leq C_0 \|[\![ d^\prime(\hat{\vartheta}) \partial_\nu \hat{\vartheta} ]\!]\|_{\fc} \|u\|_{\zfb}.$$
The remaining terms in $\mathbb{K}_3(\hat{z})w$ can be estimates in an analogous way.


For $\|\mathbb{K}_1(\hat{z}) \|_{\L(\zej,\fa)}$, one verifies by direct computation that
\begin{align*}
\|\mathbb{K}_1(\hat{z})w \|_{\fa}=&\|\hat{F}^\prime(\hat{z})w+ \hat{d}_0 M_2(\hat{h}_0):\nabla^2 u\|_{\fa}\\
&\leq C_2\|w\|_{\ej}+C_3\|w\|_{\mathbb{C}_1(J)\times\cb}.
\end{align*}
Here the constants $C_2$ and $C_3$ tend to zero as $T\to 0$.
\begin{prop}
Let $p>m+3$, $\sigma>0$. Suppose that $d_i\in C^2(0,\infty)$, $\gamma,\psi_i\in C^3(0,\infty)$. Then there exists some constant $\tau_0$ such that given any  $T\leq \tau_0$, on $J=[0,T]$, we have
\begin{align*}
\|\mathbb{L}^{-1}(\hat{\vartheta}_0)\|_{\L(\zfj,\zej)}\|\mathbb{K}^\prime(\hat{z}_0)\|_{\L(\zej,\zfj)}\leq 1/2.
\end{align*}
\end{prop}
It follows from a Neuman series argument that 
$$D_1\Phi(\bar{z},0)\in \Lis(\zej,\zfj).$$
Employing now the implicit function theorem, we attain 
$$[(\lambda,\mu,\eta)\mapsto \bar{z}_{\lambda,\mu,\eta}]\in C^k (\Br,\zej) .$$
It follows then from Theorem~\ref{S3: main thm} that
$$\bar{z}  \in C^k(\mathring{J}\times \T_{a/3}\setminus\M)\times C^k(\mathring{J}\times\M).$$
\begin{remark}
\label{S5: Proof of main RMK}
By a similar argument to the proof given at the end of Section~4.2, we can show that
$$\bar{z} \in C^k(\mathring{J}\times \Omega\setminus\M) \times C^k (\mathring{J}\times\M). $$
Together with \eqref{S5: ana of z*}, we thus have
$$\hat{z} \in C^k(\mathring{J}\times \Omega\setminus\M) \times C^k (\mathring{J}\times\M). $$
For $k\in\N\cup\{\infty\}$, since the Hanzawa transformation is $C^\infty$-smooth, the above assertion  implies that the solution $(\theta, \Gamma)$ to \eqref{stefan} has the same regularity. 
\smallskip\\
But when $k=\omega$, analyticity of the temperature $\theta$, in general, cannot be attained by applying the Hanzawa transformation.
\qed
\end{remark}

\end{document}